\documentclass{amsart}
\usepackage{amssymb}
\usepackage{mathrsfs}
\usepackage{cases}
\usepackage{amsmath}
\usepackage{amsfonts}
\usepackage{pifont}
\usepackage{graphicx,amssymb,mathrsfs,amsmath}
\usepackage{tocvsec2}
\newtheorem{thm}{Theorem}[section]

\newtheorem{lem}[thm]{Lemma}
\newtheorem{prop}[thm]{Proposition}
\theoremstyle{definition}
\newtheorem{defn}[thm]{Definition}
\newtheorem{example}[thm]{Example}
\theoremstyle{remark}
\newtheorem{rem}[thm]{Remark}
\numberwithin{equation}{section}
\begin{document}
\title[Complex powers of multivalued linear...]{Complex powers of multivalued linear operators with polynomially bounded $C$-resolvent}

\author{Marko Kosti\' c}
\address{Faculty of Technical Sciences,
University of Novi Sad,
Trg D. Obradovi\' ca 6, 21125 Novi Sad, Serbia}
\email{markokostic121@yahoo.com}

{\renewcommand{\thefootnote}{} \footnote{
\\ \text{  }  \ \    {\it Key words and phrases.} Complex powers of multivalued linear operators,
$C$-resolvent sets, abstract incomplete fractional differential inclusions,
abstract incomplete differential inclusions of second order,
locally convex spaces.
\\  \text{  } The author is partially supported by grant 174024 of Ministry
of Science and Technological Development, Republic of Serbia.}}

\begin{abstract}
In the paper under review, we construct complex powers of multivalued linear operators with polynomially bounded $C$-resolvent existing on an appropriate region of the complex plane containing the interval $(-\infty,0].$
In our approach, the operator $C$ is not necessarily injective. We clarify
the basic properties of introduced powers and analyze the abstract incomplete fractional differential inclusions
associated with the use of modified Liuoville right-sided derivatives.
We
also consider abstract  incomplete differential inclusions of second order, working in the general setting of sequentially complete locally convex spaces. Our results seem to be completely new even in the Banach space setting.
\end{abstract}
\maketitle

\section{Introduction and preliminaries}

Chronologically, the first results about fractional powers of non-negative multivalued linear operators was given by
El H. Alaarabiou \cite{bateks1}-\cite{bateks2} in 1991. In these papers, he extended the well known Hirsch functional calculus
to the class ${\mathcal M}$ of non-negative multivalued linear operators in a complex Banach space. Unfortunately,
the method proposed in \cite{bateks1}-\cite{bateks2} had not allowed one to consider
the product formula and the spectral mapping theorem for powers. Nine years later, in 2000,
C. Mart\' inez, M. Sanz and J. Pastor \cite{mart-MSP} improved a functional calculus established in \cite{bateks1}-\cite{bateks2}, providing a
new definition of fractional powers. A very stable and
consistent theory of fractional powers of the operators belonging to the class ${\mathcal M}$
has been constructed, including within itself the above-mentioned product formula, spectral mapping theorem,
as well as almost all other fundamental properties of fractional powers of non-negative
single-valued linear operators. Some later contributions have been given by J. Pastor \cite{pastor},
who considered relations between
the multiplicativity and uniqueness of fractional powers of non-negative multivalued linear operators.

The first applications
of results from the theory of multivalued linear operators to abstract degenerate differential equations were given by
A. Yagi (\cite{yagi}, 1991). In his well-known joint monograph with A. Favini \cite{faviniyagi},
the class of multivalued linear operators ${\mathcal A},$ acting on a complex Banach space
$(X,\| \cdot \|),$ for which
$(-\infty,0]\subseteq \rho({\mathcal A})$ and there exist finite numbers $M_{1}\geq 1,$ $\beta \in (0,1]$ such that
\begin{align}\label{fracvt-power-t-power}
\| R(\lambda : {\mathcal A}) \| \leq M_{1}\bigl(1+|\lambda|\bigr)^{-\beta},\quad \lambda \leq 0,
\end{align}
has been thoroughly analyzed. Assuming that (\ref{fracvt-power-t-power}) is true,
the usual von Neumann's expansion argument shows that there exist positive real constants
$c>0$ and $M>0$
such that
the resolvent set of ${\mathcal A}$ contains an open region $\Omega_{c,M} \supseteq \Omega_{c,M}' :=\{
\lambda \in {\mathbb C} : |\Im \lambda| \leq  (2M)^{-1} (c-\Re \lambda)^{\beta},\ \Re \lambda \leq c\},$
where we have the estimate
$
\| R(\lambda : {\mathcal A}) \| =O((1+|\lambda|)^{-\beta}),$ $ \lambda \in \Omega_{c,M} .$
Let $\Gamma '$ be the upwards oriented curve
$\{ \xi \pm  i(2M)^{-1}(c-\xi)^{\beta} : -\infty <\xi \leq c \}.$ In \cite{faviniyagi},
A. Favini and A. Yagi define the fractional power $
{\mathcal A}^{-\theta},$ for $\Re \theta >1-\beta ,$ by
$$
{\mathcal A}^{-\theta}:=\frac{1}{2\pi i}\int_{\Gamma'}\lambda^{-\theta} \bigl( \lambda-{\mathcal A} \bigr)^{-1}\,
d\lambda ,
$$
${\mathcal A}^{\theta}:=({\mathcal A}^{-\theta})^{-1}$ ($\Re \theta >1-\beta$); then
${\mathcal A}^{-\theta}\in L(E)$ for $\Re \theta >1-\beta ,$ and the semigroup properties
${\mathcal A}^{-\theta_{1}}{\mathcal A}^{-\theta_{2}}={\mathcal A}^{-(\theta_{1}+\theta_{2})},$ 
${\mathcal A}^{\theta_{1}}{\mathcal A}^{\theta_{2}}={\mathcal A}^{\theta_{1}+\theta_{2}}$ of powers
hold for $\Re \theta_{1},\ \Re \theta_{2}>1-\beta .$
The case $\beta \in (0,1)$ occurs in many applications and
then we cannot define satisfactorily the power
${\mathcal A}^{\theta}$ for $|\Re \theta| \leq 1-\beta .$ As explained in the introductory part of paper
\cite{favaron} by A. Favaron and A. Favini, the method of closed extensions used in the pioneering works \cite{balak} by
A. Balakrishnan
and \cite{komi} by H. Komatsu cannot be used here for construction of power ${\mathcal A}^{\theta}$
($\Re \theta \in (0,1-\beta)$). In this place, we would like to observe that the method proposed by
F. Periago, B. Straub \cite{pb1} and C. Martinez, M. Sanz, A. Redondo \cite{martinez1} (cf. also \cite{fatih})
cannot be of any help for construction of power ${\mathcal A}^{\theta}$ ($\Re \theta \in (0,1-\beta)$), as well.
In \cite[Section 9]{favaron}, the fractional power ${\mathcal A}^{\theta}$ has been constructed for
$|\Re \theta| \leq 1-\beta,$ provided the validity of condition \cite[(H3)]{favaron}.
In general case $\beta \in (0,1),$ the condition (H3) does not hold.

In order to motivate our research, assume that $\alpha \geq -1$ and a closed multivalued linear operator ${\mathcal A}$ satisfies:
\begin{center}
\begin{itemize}
\item[($\lozenge $)] $(0,\infty) \subseteq \rho({\mathcal A})$ and
\item[($\lozenge \lozenge$)] $\sup_{\lambda >0}(1+|\lambda|)^{-\alpha}||R(\lambda :
{\mathcal A})|| <\infty .$
\end{itemize}
\end{center}
Given $\beta \geq -1,$ $\varepsilon \in (0, 1],$ $d\in (0,1],$ $c'\in
(0,1)$ and $\theta \in (0,\pi],$ put $B_{d}:=\{z
\in \mathbb{C} : |z|\leq d\},$ $\Sigma_{\theta}:=\{ z\in {\mathbb C}
: z\neq 0,\ \arg (z) \in (-\theta, \theta)  \}$ and $P_{\beta,
\varepsilon, c'}:=\{ \xi +i\eta : \xi \geq \varepsilon,\ \eta \in
{\mathbb R},\ \ |\eta|\leq c'( 1+\xi )^{-\beta}\}.$ Then
the usual series argument yields that the
hypotheses ($\lozenge $)-($\lozenge \lozenge$) imply the existence
of numbers $d\in (0,1],$ $c\in (0,1),$ $\varepsilon \in (0,1]$ and $M>0$ such
that:
\begin{itemize}
\item[$(\S)$] $P_{\alpha, \varepsilon, c} \cup B_{d} \subseteq \rho({\mathcal A}),$
$(\varepsilon,c(1+\varepsilon)^{-\alpha}) \in \partial B_{d}$ and
\item[($\S \S$)] $||R(\lambda : {\mathcal A})||\leq M(1+|\lambda|)^{\alpha}, \ \lambda \in
P_{\alpha, \varepsilon, c} \cup B_{d}.$
\end{itemize}

Suppose now that $X$ is a Hausdorff sequentially complete
locally convex space over the field of complex numbers, SCLCS for short.
If $Y$ is also an SCLCS, then we denote by $L(X,Y)$ the space consisting of all continuous linear mappings from $X$ into
$Y;$ $L(X)\equiv L(X,X).$ By $\circledast_{X}$ ($\circledast$, if there is no risk for confusion), we denote the fundamental system of seminorms which defines the topology of $X.$

Keeping in mind the above analysis (the notion will be explained a little bit later), it seems reasonable to introduce the following condition:
\begin{itemize}
\item[(H)$_{0}:$] Let $C\in L(X)$ be not necessarily injective, let ${\mathcal A}$ be closed, and let $C{\mathcal A}\subseteq {\mathcal A}C.$
There exist real numbers $d\in (0,1],$ $c\in (0,1),$ $\varepsilon \in (0,1]$ and $\alpha \geq -1$ such
that $P_{\alpha, \varepsilon, c} \cup B_{d} \subseteq \rho_{C}({\mathcal A}),$ the operator family $\{ (1+|\lambda|)^{-\alpha}(\lambda - {\mathcal A})^{-1}C : \lambda \in P_{\alpha, \varepsilon, c} \cup B_{d}\} \subseteq L(X)$ is equicontinuous,
the mapping $\lambda \mapsto (\lambda - {\mathcal A})^{-1}C$ is strongly analytic on int$(P_{\alpha, \varepsilon, c} \cup B_{d})$ and strongly continuous on $\partial (P_{\alpha, \varepsilon, c} \cup B_{d}).$
\end{itemize}

The first aim of this paper is construction of complex power
$(-{\mathcal A})_{b},$ $b\in {\mathbb C}$ of a multivalued
linear operator
${\mathcal A}$ satisfying the condition (H)$_{0}.$ Although very elegant and elementary, our construction has some serious
disadvantages because
the introduced powers behave very badly (for example, we cannot expect the additivity property of powers clarified in \cite[Remark 2.11]{TAJW}) in the case that the regularizing operator $C_{1},$ defined below, is not injective (since the resolvents and $C$-resolvents of a really multivalued linear operator are not injective, this is the main case in our considerations).
The method proposed
for construction of power $(-{\mathcal A})_{b}$ is
different from that already employed in single-valued linear case \cite{complex-new-single};
in this paper, we first apply regularization with the operator $C_{1}$
and follow after that the
approach from our joint research paper with C. Chen, M. Li and M. \v Zigi\' c \cite{TAJW}.
In particular, we define any complex power of a
multivalued linear operator satisfying (\ref{fracvt-power-t-power}) and not (H3).

The following sectorial analogue
of (H) is most important in applications:
\begin{itemize}
\item[(HS)$_{0}:$] Let $C\in L(X)$ be not necessarily injective, let ${\mathcal A}$ be closed, and let $C{\mathcal A}\subseteq {\mathcal A}C.$
There exist real numbers $d\in (0,1],$ $\vartheta \in (0,\pi/2)$ and $\alpha \geq -1$ such
that $\Sigma_{\vartheta} \cup B_{d} \subseteq \rho_{C}({\mathcal A}),$ the operator family $\{ (1+|\lambda|)^{-\alpha}(\lambda - {\mathcal A})^{-1}C : \lambda \in \Sigma_{\vartheta} \cup B_{d}\} \subseteq L(X)$ is equicontinuous,
the mapping $\lambda \mapsto (\lambda - {\mathcal A})^{-1}C$ is strongly analytic on int$(\Sigma_{\vartheta}\cup B_{d})$ and strongly continuous on $\partial (\Sigma_{\vartheta} \cup B_{d}).$
\end{itemize}

The construction of power ($-{\mathcal A})_{b},$ $b\in {\mathbb C}$ of a multivalued linear operator ${\mathcal A}$
for which
$0\notin \text{int}(\rho_{C}({\mathcal A}))$ is without the scope of this paper. For the sake of brevity and better exposition,
we will not compare the construction presented here with those appearing in \cite{bateks1}-\cite{bateks2},
\cite{balak}, \cite{fatih}, \cite{ralfinjo}, \cite{favaron}, \cite{komi}, \cite{complex-new-single}, \cite{mart-MSP}, \cite{pb1} and
\cite{hir}, if it makes any sense for doing so.

The second aim of this paper is to show that a great number of resolvent equations and
generalized resolvent equations holds for $C$-resolvents of multivalued linear operators,
where $C$ is non-injective, in general.
The third and, simultaneously, the main aim of this paper is to continue our recent research study \cite{inc}
of abstract incomplete fractional degenerate differential equations with modified Liouville right-sided fractional
derivatives \cite{knjigaho} and abstract incomplete degenerate differential equations of second order.
We consider fractionally integrated $C_{1}$-regularized semigroups generated by the negatives of introduced powers, and provide a few relevant applications of our theoretical results to abstract incomplete degenerate PDEs.

The organization of material is briefly described as follows. In the second section, we collect the basic definitions and results
from the theory of multivalued linear operators that are necessary for our further work; in
two separate subsections, we consider
$C$-resolvents of multivalued linear operators and generation of fractionally integrated $C$-semigroups by
multivalued linear operators. In the third section, we define the complex powers of operators satisfying
(H)$_{0}$ or (HS)$_{0}$ and clarify their most intriguing properties. The fourth section of paper
is completely devoted to the study of abstract incomplete differential inclusions.

We use the standard notation throughout the paper.
Unless specifed otherwise,
we assume
that $X$ is a Hausdorff sequentially complete
locally convex space over the field of complex numbers, SCLCS for short.
Let ${\mathcal B}$ be the family consisting of all bounded subsets\index{bounded subset} of $X,$ and
let $p_{{\mathbb B}}(T):=\sup_{x\in {\mathbb B}}p(Tx),$ $p\in \circledast_{X},$ ${\mathbb B}\in
{\mathcal B},$ $T\in L(X).$ Then $p_{{\mathbb B}}(\cdot)$ is a seminorm\index{seminorm} on
$L(X)$ and the system $(p_{{\mathbb B}})_{(p,{\mathbb B})\in \circledast_{X} \times
{\mathcal B}}$ induces the Hausdorff locally convex topology on
$L(X).$
If $X$ is a Banach space, then we denote by $\|x\|$ the norm of an element $x\in X.$
The Hausdorff locally convex topology on $X^{\ast},$ the dual space\index{dual space} of $X,$
defines the
system $(|\cdot|_{{\mathbb B}})_{{\mathbb B}\in {\mathcal B}}$ of seminorms on
$X^{\ast},$ where and in the sequel $|x^{\ast}|_{{\mathbb B}}:=\sup_{x\in
{\mathbb B}}|\langle x^{\ast}, x \rangle |,$ $x^{\ast} \in X^{\ast},$ ${\mathbb B} \in
{\mathcal B}.$
Let us recall that the spaces $L(X)$ and $X^{\ast}$ are sequentially
complete provided that $X$ is barreled\index{barreled space} (\cite{meise}). Set
$g_{\zeta}(t):=t^{\zeta-1}/\Gamma(\zeta),$ $\lfloor s \rfloor:=\sup\{n\in {\mathbb Z}:\allowbreak n\leq s\},$
$\lceil s\rceil :=\inf\{n\in {\mathbb Z}: s\leq n\}$ ($\zeta>0,$ $s\in{\mathbb R}$) and ${\mathbb C}_{+}:=\{z\in {\mathbb C} : \Re z>0\};$ here, $\Gamma(\cdot)$ denotes the Gamma function.
Let $\delta_{\cdot,\cdot}$ denote the Kronecker delta.

If $V$ is a general topological vector space,
then a function $f :
\Omega \rightarrow V,$ where $\Omega$ is an open subset of ${\mathbb
C},$ is said to be analytic if it is locally expressible in a
neighborhood of any point $z\in \Omega$ by a uniformly convergent
power series with coefficients in $V.$
We refer the reader to \cite[Section 1.1]{knjigaho} and references cited there for the basic information about vector-valued analytic functions.
In our approach the space $X$ is sequentially complete, so that the analyticity of a mapping
$f: \Omega \rightarrow X$ ($\emptyset \neq \Omega \subseteq {\mathbb C}$) is equivalent with its weak analyticity.

The integration of functions with values in SCLCSs is still an active field of research. In  this paper, we follow the approach used in the monograph \cite{martinez} by C. Martinez and M. Sanz, see pp. 99-102.
Concerning the Laplace transform of functions with values in SCLCSs, we refer the reader to \cite{FKP} and \cite{x263}; cf. \cite{a43}
for the Banach space case. Fairly complete information on abstract degenerate differential equations can be obtained by consulting the monographs \cite{carol}, \cite{dem}, \cite{faviniyagi}, \cite{FKP}, \cite{me152} and \cite{svir-fedorov}.

Considerable interest in fractional calculus and fractional differential equations
has been stimulated due to their numerous applications in engineering, physics, chemistry, biology and other sciences (see e.g. \cite{bajlekova}, \cite{Diet}, \cite{knjigah} and
\cite{Po}-\cite{samko}).
In this paper, we use the modified
Liouville right-sided fractional
derivatives. Suppose that $\beta>0$ and $\beta \notin {\mathbb N}.$
Then the Liouville right-sided fractional
derivative of order $\beta$ (see \cite[(2.3.4)]{kilbas} for the scalar-valued case) is defined for those continuous
functions\index{fractional derivatives!Liouville right-sided}
$u : (0,\infty) \rightarrow X$ for which $\lim_{T\rightarrow \infty}\int_{s}^{T}g_{\lceil \beta \rceil -\beta}(t-s)u(t)\, dt=\int_{s}^{\infty}g_{\lceil \beta \rceil -\beta}(t-s)u(t)\, dt$
exists and defines a $\lceil \beta \rceil$-times continuously differentiable function on $(0,\infty)$, by
$$
{\bf D}_-^\beta u(s):=(-1)^{\lceil \beta \rceil}\frac{d^{\lceil \beta \rceil}}{ds^{\lceil \beta \rceil}}\int \limits_{s}^{\infty}g_{\lceil \beta \rceil -\beta}(t-s)u(t)\, dt,\quad s>0.
$$
We define the modified
Liouville right-sided fractional
derivative of order $\beta,$ $D^{\beta}_{-}u(s)$ shortly, for those continuously differentiable
functions\index{fractional derivatives!Liouville right-sided}
$u : (0,\infty) \rightarrow X$ for which $\lim_{T\rightarrow \infty}\int_{s}^{T}g_{\lceil \beta \rceil -\beta}(t-s)u^{\prime}(t)\, dt=\int_{s}^{\infty}g_{\lceil \beta \rceil -\beta}(t-s)u^{\prime}(t)\, dt$
exists and defines a $\lceil \beta -1\rceil$-times continuously differentiable function on $(0,\infty)$, by
$$
D_-^\beta u(s):=(-1)^{\lceil \beta \rceil}\frac{d^{\lceil \beta -1 \rceil}}{ds^{\lceil \beta -1\rceil}}\int \limits_{s}^{\infty}g_{\lceil \beta \rceil -\beta}(t-s)u^{\prime}(t)\, dt,\quad s>0;
$$
if $\beta =n\in {\mathbb N},$ then ${\bf D}^{n}_{-}u$ and $D^{n}_{-}u$ are defined for all $n$-times continuously differentiable functions $u(\cdot)$ on $(0,\infty),$
by ${\bf D}^{n}_{-}u:=D^{n}_{-}u:= (-1)^{n}d/d^{n},$ where $d/d^{n}$ denotes the usual
derivative operator of order $n$ (cf. also \cite[(2.3.5)]{kilbas}).

\section{Multivalued linear operators}

In this section, we will present some necessary definitions from the theory of multivalued
linear operators. For more details about this topic, we refer the reader to the monographs by R. Cross \cite{cross} and A. Favini, A. Yagi \cite{faviniyagi}.

Let $X$ and $Y$ be two sequentially complete locally convex spaces over the field of complex numbers.
A multivalued map ${\mathcal A} : X \rightarrow P(Y)$ is said to be a multivalued
linear operator (MLO) iff the following holds:
\begin{itemize}
\item[(i)] $D({\mathcal A}) := \{x \in X : {\mathcal A}x \neq \emptyset\}$ is a linear subspace of $X$;
\item[(ii)] ${\mathcal A}x +{\mathcal A}y \subseteq {\mathcal A}(x + y),$ $x,\ y \in D({\mathcal A})$
and $\lambda {\mathcal A}x \subseteq {\mathcal A}(\lambda x),$ $\lambda \in {\mathbb C},$ $x \in D({\mathcal A}).$
\end{itemize}
If $X=Y,$ then we say that ${\mathcal A}$ is an MLO in $X.$
As an almost immediate consequence of definition, we have that the equality $\lambda {\mathcal A}x + \eta {\mathcal A}y = {\mathcal A}(\lambda x + \eta y)$ holds
for every $x,\ y\in D({\mathcal A})$ and for every $\lambda,\ \eta \in {\mathbb C}$ with $|\lambda| + |\eta| \neq 0.$ If ${\mathcal A}$ is an MLO, then ${\mathcal A}0$ is a linear manifold in $Y$
and ${\mathcal A}x = f + {\mathcal A}0$ for any $x \in D({\mathcal A})$ and $f \in {\mathcal A}x.$ Define $R({\mathcal A}):=\{{\mathcal A}x :  x\in D({\mathcal A})\}.$
Then the set $N({\mathcal A}):={\mathcal A}^{-1}0 = \{x \in D({\mathcal A}) : 0 \in {\mathcal A}x\}$ is called the kernel
of ${\mathcal A}.$ The inverse ${\mathcal A}^{-1}$ of an MLO is defined by
$D({\mathcal A}^{-1}) := R({\mathcal A})$ and ${\mathcal A}^{-1} y := \{x \in D({\mathcal A}) : y \in {\mathcal A}x\}$.
It can be easily seen that ${\mathcal A}^{-1}$ is an MLO in $X,$ as well as that $N({\mathcal A}^{-1}) = {\mathcal A}0$
and $({\mathcal A}^{-1})^{-1}={\mathcal A}.$ If $N({\mathcal A}) = \{0\},$ i.e., if ${\mathcal A}^{-1}$ is
single-valued, then ${\mathcal A}$ is said to be injective. If ${\mathcal A},\ {\mathcal B} : X \rightarrow P(Y)$ are two MLOs, then we define its sum ${\mathcal A}+{\mathcal B}$ by $D({\mathcal A}+{\mathcal B}) := D({\mathcal A})\cap D({\mathcal B})$ and $({\mathcal A}+{\mathcal B})x := {\mathcal A}x +{\mathcal B}x,$ $x\in D({\mathcal A}+{\mathcal B}).$
It is clear that ${\mathcal A}+{\mathcal B}$ is likewise an MLO. We write ${\mathcal A} \subseteq {\mathcal B}$ iff $D({\mathcal A}) \subseteq D({\mathcal B})$ and ${\mathcal A}x \subseteq {\mathcal B}x$
for all $x\in D({\mathcal A}).$

Let ${\mathcal A} : X \rightarrow P(Y)$ and ${\mathcal B} : Y\rightarrow P(Z)$ be two MLOs, where $Z$ is an SCLCS. The product of ${\mathcal A}$
and ${\mathcal B}$ is defined by $D({\mathcal B}{\mathcal A}) :=\{x \in D({\mathcal A}) : D({\mathcal B})\cap {\mathcal A}x \neq \emptyset\}$ and
${\mathcal B}{\mathcal A}x:=
{\mathcal B}(D({\mathcal B})\cap {\mathcal A}x).$ Then ${\mathcal B}{\mathcal A} : X\rightarrow P(Z)$ is an MLO and
$({\mathcal B}{\mathcal A})^{-1} = {\mathcal A}^{-1}{\mathcal B}^{-1}.$ The scalar multiplication of an MLO ${\mathcal A} : X\rightarrow P(Y)$ with the number $z\in {\mathbb C},$ $z{\mathcal A}$ for short, is defined by
$D(z{\mathcal A}):=D({\mathcal A})$ and $(z{\mathcal A})(x):=z{\mathcal A}x,$ $x\in D({\mathcal A}).$ It is clear that $z{\mathcal A}  : X\rightarrow P(Y)$ is an MLO and $(\omega z){\mathcal A}=\omega(z{\mathcal A})=z(\omega {\mathcal A}),$ $z,\ \omega \in {\mathbb C}.$

The integer powers of an MLO ${\mathcal A} :  X\rightarrow P(X)$ is defined recursively as follows: ${\mathcal A}^{0}=:I;$ if ${\mathcal A}^{n-1}$ is defined, set\index{multivalued linear operator!integer powers}
$$
D({\mathcal A}^{n}) := \bigl\{x \in  D({\mathcal A}^{n-1}) : D({\mathcal A}) \cap {\mathcal A}^{n-1}x \neq \emptyset \bigr\},
$$
and
$$
{\mathcal A}^{n}x := \bigl({\mathcal A}{\mathcal A}^{n-1}\bigr)x =\bigcup_{y\in  D({\mathcal A}) \cap {\mathcal A}^{n-1}x}{\mathcal A}y,\quad x\in D( {\mathcal A}^{n}).
$$
We can prove inductively that $({\mathcal A}^{n})^{-1} = ({\mathcal A}^{n-1})^{-1}{\mathcal A}^{-1} = ({\mathcal A}^{-1})^{n}=:{\mathcal A}^{-n},$ $n \in {\mathbb N}$
and $D((\lambda-{\mathcal A})^{n})=D({\mathcal A}^{n}),$ $n \in {\mathbb N}_{0},$ $\lambda \in {\mathbb C}.$ Moreover,
if ${\mathcal A}$ is single-valued, then the above definitions are consistent with the usual definitions of powers of ${\mathcal A}.$

We say that an MLO ${\mathcal A} : X\rightarrow P(Y)$ is closed if for any
nets $(x_{\tau})$ in $D({\mathcal A})$ and $(y_{\tau})$ in $Y$ such that $y_{\tau}\in {\mathcal A}x_{\tau}$ for all $\tau\in I$ we have that the suppositions $\lim_{\tau \rightarrow \infty}x_{\tau}=x$ and
$\lim_{\tau \rightarrow \infty}y_{\tau}=y$ imply
$x\in D({\mathcal A})$ and $y\in {\mathcal A}x$ (cf. \cite{meise} for the notion). As it is well-known any MLO ${\mathcal A}$ is closable and its closure $\overline{{\mathcal A}}$ is a closed MLO. 

Suppose that ${\mathcal A} : X\rightarrow P(Y)$ is an MLO. Then we define the adjoint ${\mathcal A}^{\ast}: Y^{\ast}\rightarrow P(X^{\ast})$\index{multivalued linear operator!adjoint}
of ${\mathcal A}$ by its graph
$$
{\mathcal A}^{\ast}:=\Bigl\{ \bigl( y^{\ast},x^{\ast}\bigr)  \in Y^{\ast} \times X^{\ast} :  \bigl\langle y^{\ast},y \bigr\rangle =\bigl \langle x^{\ast}, x\bigr \rangle \mbox{ for all pairs }(x,y)\in {\mathcal A} \Bigr\}.
$$
It is simpy verified that ${\mathcal A}^{\ast}$
is a closed MLO, and that $ \langle y^{\ast},y \rangle =0$ whenever $y^{\ast}\in D({\mathcal A}^{\ast})$ and $y\in {\mathcal A}0.$
It can be easily checked that the equations \cite[(1.2)-(1.6)]{faviniyagi} continue to hold for adjoints of MLOs acting on locally convex spaces.

We need the following auxiliary lemma from \cite{FKP}.

\begin{lem}\label{integracija-tricky}
Let $\Omega$ be a locally compact, separable metric space,
and let $\mu$ be a locally finite
Borel measure defined on $\Omega.$
Suppose that ${\mathcal A} : X\rightarrow P(Y)$ is a closed \emph{MLO.} Let $f : \Omega \rightarrow X$ and $g : \Omega \rightarrow Y$ be $\mu$-integrable, and let $g(x)\in {\mathcal A}f(x),$ $x\in \Omega.$ Then $\int_{\Omega}f\, d\mu \in D({\mathcal A})$ and $\int_{\Omega}g\, d\mu\in {\mathcal A}\int_{\Omega}f\, d\mu.$
\end{lem}

\subsection{$C$-Resolvent sets of multivalued linear operators}
In this subsection, we consider the $C$-resolvent sets of MLOs in locally convex spaces.\index{multivalued linear operator!$C$-resolvent} Our standing assumptions is that ${\mathcal A}$ is an MLO in $X$, as well as that $C\in L(X)$ and $C{\mathcal A}\subseteq {\mathcal A}C$ (this is equivalent to say that, for any $(x,y)\in X\times X,$ we have the implication $(x,y)\in {\mathcal A} \Rightarrow (Cx,Cy)\in {\mathcal A};$ then $C{\mathcal A}^{k}\subseteq {\mathcal A}^{k}C$ for all $k\in {\mathbb N}$). Here it is worth noting that we do not require the injectiveness of operator $C$ (cf. \cite{FKP} for more details concerning this case).
Then
the $C$-resolvent set of ${\mathcal A},$ $\rho_{C}({\mathcal A})$ for short, is defined as the union of those complex numbers
$\lambda \in {\mathbb C}$ for which
\begin{itemize}
\item[(i)] $R(C)\subseteq R(\lambda-{\mathcal A})$;
\item[(ii)] $(\lambda - {\mathcal A})^{-1}C$ is a single-valued linear continuous operator on $X.$
\end{itemize}
The operator $\lambda \mapsto (\lambda -{\mathcal A})^{-1}C$ is called the $C$-resolvent of ${\mathcal A}$ ($\lambda \in \rho_{C}({\mathcal A})$); the resolvent set of ${\mathcal A}$ is defined by $\rho({\mathcal A}):=\rho_{I}({\mathcal A}),$ $R(\lambda : {\mathcal A})\equiv  (\lambda -{\mathcal A})^{-1}$  ($\lambda \in \rho({\mathcal A})$).
The basic properties of $C$-resolvent sets of single-valued linear operators (\cite{knjigah}-\cite{knjigaho}) continue to hold in our framework; for instance, if $\rho({\mathcal A})\neq \emptyset,$ then ${\mathcal A}$ is closed; it is well known that this statement does not hold if $\rho_{C}({\mathcal A})\neq \emptyset$ for some (injective, non-injective) $C\neq I$ (cf. \cite[Example 2.2]{ralfinjo}). 

\begin{thm}\label{C-favini}
\begin{itemize}
\item[(i)]
We have
$$
\bigl( \lambda-{\mathcal A} \bigr)^{-1}C{\mathcal A}\subseteq \lambda \bigl( \lambda-{\mathcal A} \bigr)^{-1}C-C\subseteq {\mathcal A}\bigl( \lambda-{\mathcal A} \bigr)^{-1}C,\quad \lambda \in \rho_{C}({\mathcal A}).
$$
The operator $( \lambda-{\mathcal A})^{-1}C{\mathcal A}$ is single-valued on $D({\mathcal A})$ and $( \lambda-{\mathcal A})^{-1}C{\mathcal A}x=( \lambda-{\mathcal A})^{-1}Cy,$ whenever $y\in {\mathcal A}x$ and $\lambda \in \rho_{C}({\mathcal A}).$
\item[(ii)] Suppose that $\lambda,\ \mu \in \rho_{C}({\mathcal A}).$ Then the resolvent equation\index{multivalued linear operator!resolvent equation}
$$
\bigl( \lambda-{\mathcal A}\bigr)^{-1}C^{2}x-\bigl( \mu-{\mathcal A}\bigr)^{-1}C^{2}x=(\mu -\lambda)\bigl( \lambda-{\mathcal A}\bigr)^{-1}C\bigl( \mu-{\mathcal A}\bigr)^{-1}Cx,\quad x\in X
$$
holds good. In particular, $( \lambda-{\mathcal A})^{-1}C( \mu-{\mathcal A})^{-1}C=( \mu-{\mathcal A})^{-1}C (\lambda-{\mathcal A}\bigr)^{-1}C.$
\end{itemize}
\end{thm}

\begin{proof}
Albeit the proof can be deduced by using the argumentation contained in those of \cite[Theorem 1.7-Theorem 1.9]{faviniyagi}, we will include all relevant details. For the first inclusion in (i), let us assume that $\lambda \in \rho_{C}({\mathcal A})$ and  $(x,y)\in (\lambda -{\mathcal A})^{-1}C{\mathcal A},$ i.e., there exists $z\in {\mathcal A}x$ such that $y=(\lambda -{\mathcal A})^{-1}Cz.$ This implies $Cz \in \lambda y-{\mathcal A}y,$ $\lambda y \in Cz+{\mathcal A}y,$ and since $C{\mathcal A}\subseteq {\mathcal A}C,$ $\lambda y \in {\mathcal A}Cx+{\mathcal A}y={\mathcal A}(Cx+y).$ Hence, $\lambda Cx\in (\lambda -{\mathcal A})y+(\lambda-{\mathcal A})Cx=\lambda y +\lambda Cx-{\mathcal A}(Cx+y)$ so that $y+Cx=\lambda ( \lambda-{\mathcal A})^{-1}Cx$ and $y=\lambda ( \lambda-{\mathcal A})^{-1}Cx-Cx,$ as required. For the second inclusion in (i), let  $\lambda \in \rho_{C}({\mathcal A})$ and $y=( \lambda-{\mathcal A})^{-1}Cx.$ Then $\lambda y-Cx\in {\mathcal A}y$ so that $\lambda ( \lambda-{\mathcal A})^{-1}Cx-Cx\in {\mathcal A}( \lambda-{\mathcal A})^{-1}Cx.$ The Hilbert resolvent equation in (ii) can be shown as follows. Let $\lambda,\ \mu \in \rho_{C}({\mathcal A}).$ Using the second inclusion in (i), we get that
\begin{align*}
\Bigl( & (\lambda -\nu) -({\mathcal A} -\nu) \Bigr)^{-1}C({\mathcal A} -\nu)\bigl( {\mathcal A} -\nu \bigr)^{-1}Cx
\\ & = (\lambda -\nu)\Bigl( (\lambda -\nu) -({\mathcal A} -\nu) \Bigr)^{-1}C\bigl( {\mathcal A} -\nu \bigr)^{-1}Cx-\bigl({\mathcal A}-\nu \bigr)^{-1}C^{2}x,
\end{align*}
for any $x\in X.$ Since the operator $( (\lambda -\nu)- ({\mathcal A}-\nu))^{-1}C({\mathcal A}-\nu)$ is single-valued on $D({\mathcal A}-\nu)=D({\mathcal A}),$ the result immediately follows.
\end{proof}

Having in mind Theorem \ref{C-favini}, we can simply extend the assertion of \cite[Proposition 2.1.14]{knjigaho} to MLOs in locally convex spaces (see e.g. the proofs of \cite[Proposition 2.6, Corollary 2.8]{ralfinjo} for the Banach space setting); the only thing worth noting is the following: Suppose that ${\mathcal A}$ is closed. Then card$((\lambda- {\mathcal A})^{-n}Cx)\leq 1,$ $\lambda \in \rho_{C}({\mathcal A}),$ $n\in {\mathbb N},$ $x\in X.$ This can be proved by induction, observing that $(\lambda -{\mathcal A})^{-1}0$ is a singleton ($\lambda \in \rho_{C}({\mathcal A})$) as well as that for each $y\in (\lambda- {\mathcal A})^{-(n+1)}Cx$ ($\lambda \in \rho_{C}({\mathcal A}),$ $n\in {\mathbb N},$ $x\in X$) we have
\begin{align*}
\bigl(\lambda- {\mathcal A}\bigr)^{-(n+1)}Cx& =y+\bigl(\lambda- {\mathcal A}\bigr)^{-(n+1)}C0\\ &=y+\bigl(\lambda- {\mathcal A}\bigr)^{-1}\bigl(\lambda- {\mathcal A}\bigr)^{-n}C0=y+\bigl(\lambda- {\mathcal A}\bigr)^{-1}0=\{y\}.
\end{align*}
Now we can proceed as in the proof of \cite[Corollary 2.8]{ralfinjo} in order to see that the equation (\ref{creso}) in (iii) holds, and that
 (\ref{creso1}) in (iii) holds, provided in addition that $X$ is barreled.

\begin{prop}\label{2.16-favini}
Let $\emptyset \neq \Omega \subseteq \rho_{C}({\mathcal A})$ be open, and let $x\in X.$
\begin{itemize}
\item[(i)] The local boundedness of the mapping $\lambda \mapsto
(\lambda-{\mathcal A})^{-1}Cx,$ $\lambda \in \Omega,$ resp. the assumption that
$X$ is barreled and the local boundedness of the mapping $\lambda
\mapsto (\lambda-{\mathcal A})^{-1}C,\ \lambda \in \Omega ,$ implies the
analyticity of the mapping $\lambda \mapsto (\lambda-{\mathcal A})^{-1}C^{3}x,\
\lambda \in \Omega ,$ resp. $\lambda \mapsto (\lambda-{\mathcal A})^{-1}C^{3},\
\lambda \in \Omega .$ Furthermore, if $R(C)$ is dense in $X,$
resp. if $R(C)$ is dense in $X$ and $X$ is barreled, then the
mapping $\lambda \mapsto (\lambda-{\mathcal A})^{-1}Cx,\ \lambda \in \Omega $
is analytic, resp. the mapping $\lambda \mapsto (\lambda-{\mathcal A})^{-1}C,\
\lambda \in \Omega$ is analytic.
\item[(ii)] Suppose that $R(C)$ is dense in $X.$
Then the local boundedness of the mapping $\lambda \mapsto
(\lambda-{\mathcal A})^{-1}Cx,\ \lambda \in \Omega$ implies its analyticity. Furthermore, if $X$ is barreled, then the local boundedness of the
mapping $\lambda \mapsto (\lambda-{\mathcal A})^{-1}C,\ \lambda \in \Omega$
implies its analyticity.
\item[(iii)] Suppose that $R(C)$ is dense in $X$ and ${\mathcal A}$ is closed. Then we have $R(C)\subseteq \hbox{R}((\lambda-{\mathcal A})^{n}),\ n\in {\mathbb N}$ and
\begin{equation}\label{creso}
\frac{d^{n-1}}{d\lambda^{n-1}}\bigl(\lambda-{\mathcal A} \bigr)^{-1}Cx=(-1)^{n-1}(n-1)!
\bigl(\lambda-{\mathcal A}\bigr)^{-n}Cx,\ n\in {\mathbb N}.
\end{equation}
Furthermore, if $X$ is barreled, then $R(C) \subseteq
R((\lambda-{\mathcal A})^{n}),\ n\in {\mathbb N}$ and
\begin{equation}\label{creso1}
\frac{d^{n-1}}{d\lambda^{n-1}}\bigl(\lambda- {\mathcal A} \bigr)^{-1}C=(-1)^{n-1}(n-1)!
\bigl(\lambda-{\mathcal A}\bigr)^{-n}C\in L(X),\ n\in {\mathbb N}.
\end{equation}
\end{itemize}
\end{prop}

\begin{rem}\label{univ-prof} 
Let $\emptyset \neq \Omega \subseteq \rho_{C}({\mathcal A})$ be open, and let $x\in X.$
Suppose that the operator $C$ is injective.
Then the continuity of mapping $\lambda \mapsto
(\lambda-{\mathcal A})^{-1}Cx,\ \lambda \in \Omega$ implies its analyticity
and
\begin{align*}
\frac{d^{n-1}}{d\lambda^{n-1}}\bigl(\lambda-{\mathcal A} \bigr)^{-1}Cx=(-1)^{n-1}(n-1)!
\bigl(\lambda-\overline{{\mathcal A}}\bigr)^{-n}Cx,\ n\in {\mathbb N}.
\end{align*}
Furthermore, if $X$ is barreled, then the
continuity of mapping $\lambda \mapsto (\lambda-{\mathcal A})^{-1}C,\ \lambda
\in \Omega$ implies its analyticity and
\begin{align*}
\frac{d^{n-1}}{d\lambda^{n-1}}\bigl(\lambda- {\mathcal A} \bigr)^{-1}C=(-1)^{n-1}(n-1)!
\bigl(\lambda- \overline{{\mathcal A}}\bigr)^{-n}C\in L(X),\ n\in {\mathbb N};
\end{align*}
cf. also \cite[Remark 2.7]{ralfinjo}.
\end{rem}

The following generalized resolvent formulae are very important in our work; with the help of Theorem \ref{C-favini}, their validities can be simply proved inductively (compare with the generalized resolvent formula opening the third section of paper):

\begin{thm}\label{gen-res-MLO}
\begin{itemize}
\item[(i)] Let $x\in X,$ $k\in {\mathbb N}_{0}$ and $\lambda,\ z\in
\rho_{C}({\mathcal A})$ with $z\neq \lambda .$ Then the following holds:
\begin{align*}
\bigl(z&-{\mathcal A}\bigr)^{-1}C \Bigl(\bigl(\lambda-{\mathcal A}\bigr)^{-1}
C\Bigr)^{k}x
\\ &=\frac{(-1)^{k}}{(z-\lambda)^{k}}\bigl(z-{\mathcal A}\bigr)^{-1}C^{k+1}x + \sum
\limits^{k}_{i=1}\frac{(-1)^{k-i}\bigl((\lambda-{\mathcal A})^{-1}C\bigr)^{i}C^{k+1-i}x}{\bigl(z-\lambda\bigr)^{k+1-i}}.
\end{align*}
\item[(ii)] Let $k\in {\mathbb N}_{0},$ $x,\ y\in X$, $y\in (\lambda_{0}-{\mathcal A})^{k}x$
and $\lambda_{0},\ z\in
\rho_{C}({\mathcal A})$ with $z\neq \lambda_{0} .$ Then the following holds:
\begin{align*}
\bigl( z-{\mathcal A} \bigr)^{-1}C^{k+1}x&=\frac{(-1)^{k}}{\bigl(  z-\lambda_{0} \bigr)^{k}}\bigl( z-{\mathcal A} \bigr)^{-1}C^{k+1}y
\\&+\sum \limits_{i=1}^{k}\frac{(-1)^{k-i}\bigl(( \lambda_{0}-{\mathcal A})^{-1}C \bigr)^{i}C^{k+1-i}y}{\bigl( z-\lambda_{0}\bigr)^{k+1-i}}.
\end{align*}
\end{itemize}
\end{thm}

Now we introduce the following definition.

\begin{defn}\label{c-nonnegativealm} (cf. \cite[Definition 2.9.4]{knjigaho} for single-valued case) Let ${\mathcal A}$  be a closed multivalued linear operator on $X.$
\begin{itemize}
\item[(i)] Then we say that ${\mathcal A}$ is $C$-nonnegative\index{operator!$C$-nonnegative} iff $(- \infty, 0) \subseteq \rho_C ({\mathcal A})$ and
the family
$$
\Bigl\{ \lambda \bigl(\lambda + {\mathcal A}\bigr)^{-1} C: \lambda > 0 \Bigr\}
$$
is equicontinuous; moreover, a $C$-nonnegative operator ${\mathcal A}$ is
called $C$-positive\index{operator!$C$-positive} iff, in addition, $0\in \rho_{C}({\mathcal A}).$
\item[(ii)] Let $0 \leq \omega < \pi$.
Then we say that ${\mathcal A}$ is $C$-sectorial\index{operator!$C$-sectorial} of angle $\omega,$ in short ${\mathcal A}\in
\emph{Sect}_{C}(\omega),$ iff $\mathbb{C} \setminus
\overline{\Sigma_{\omega}} \subseteq \rho_C ({\mathcal A})$ and the family
$$
\Bigl\{ \lambda \bigl(\lambda - {\mathcal A}\bigr)^{-1} C : \lambda \notin
\Sigma_{\omega '} \Bigr\}
$$
is equicontinuous for every $\omega < \omega ' < \pi;$ if this is
the case, then the $C$-spectral angle\index{$C$-spectral angle} of ${\mathcal A}$ is defined by
$\omega_{C}({\mathcal A}):=\inf\{\omega \in [0,\pi): {\mathcal A}\in
\emph{Sect}_{C}(\omega)\}.$
\end{itemize}
\end{defn}

We close this subsection by observing
that
some properties of $C$-nonnegative operators established in \cite[Proposition 2.4]{TAJW} continue to hold in multivalued linear case. For example,
we can prove that:
\begin{itemize}
\item[(i)] If $0 \in \rho (C)$, then ${\mathcal A}$ is $C$-nonnegative iff
${\mathcal A}$ is nonnegative.
\item[(ii)] If ${\mathcal A}$ is $C$-positive, then the family
$\{(\lambda + C) (\lambda + {\mathcal A})^{-1} C: \lambda > 0\}$ is
equicontinuous. Conversely, if the last family is equicontinuous and
$C$ is nonnegative, then ${\mathcal A}$ is $C$-nonnegative.
\item[(iii)] Let $X$ be barreled. Then the adjoint ${\mathcal A}^*$ of ${\mathcal A}$ is $C^{\ast}$-nonnegative in $X^{\ast}$.
\end{itemize}

\subsection{Fractionally integrated $C$-semigroups in locally convex spaces}\label{maxx}

In this subsection, we collect the basic facts and definitions about (degenerate) fractionally integrated $C$-semigroups in locally convex spaces.
The operator $C\in L(X)$ need not be injective.

\begin{defn}\label{2.1.1.1'} (\cite{catania})
Suppose that $0<\alpha <\infty$ and $0<\tau \leq \infty .$
Then a strongly continuous operator family $(S_\alpha(t))_{t\in [0,\tau)}\subseteq L(X)$ is called a (local, if $\tau<\infty$) $\alpha$-times integrated $C$-semigroup
iff the following holds:
\begin{itemize}
\item[(i)] $S_\alpha(t)C=CS_\alpha(t)$, $t\in [0,\tau),$ and
\item[(ii)] For all $x\in X$ and $t,\ s\in [0,\tau)$ with $t+s\in [0,\tau),$ we have
\begin{align*}
S_\alpha(t)S_\alpha(s)x=\Biggl[\int_0^{t+s}-\int_0^t-\int_0^s\Biggr]
g_{\alpha}(t+s-r)S_\alpha(r)Cx\,dr.
\end{align*}
\end{itemize}
\end{defn}

By a $C$-regularized semigroup ($0$-times integrated $C$-regularized semigroup) we mean any strongly continuous operator family $(S_0(t)\equiv S(t))_{t\in [0,\tau)}\subseteq L(X)$
such that $S(t)C=CS(t)$, $t\in [0,\tau)$ and $S(t+s)C=S(t)S(s)$ for all $t,\ s\in [0,\tau)$ with $t+s\in [0,\tau).$

Let $0\leq \alpha < \infty$. If $\tau=\infty ,$ then $(S_{\alpha}(t))_{t\geq 0}$ is said to be
exponentially equicontinuous\index{$(a,k)$-regularized $C$-resolvent family!exponentially equicontinuous} (equicontinuous\index{$(a,k)$-regularized $C$-resolvent family!equicontinuous}) iff there exists
$\omega \in {\mathbb R}$ ($\omega =0$) such that the family $\{
e^{-\omega t} S_{\alpha}(t) : t\geq 0\}$ is equicontinuous; $(S_{\alpha}(t))_{t\in [0,\tau)}$ is said to be
locally equicontinuous iff for every $\tau' \in (0,\tau)$ we have that the operator family $\{
S_{\alpha}(t) : t\in [0,\tau'] \}$ is equicontinuous.

The integral generator $\hat{{\mathcal A}}$ of $(S_{\alpha}(t))_{t\in [0,\tau)}$ is defined by its graph
\[
\hat{{\mathcal A}}:=\Biggl\{(x,y)\in X\times X:S_\alpha(t)x-g_{\alpha} (t)Cx=\int\limits_0^tS_\alpha(s)y\,ds,\; t\in [0,\tau) \Biggr\}.
\]
The integral generator $\hat{{\mathcal A}}$ of $(S_{\alpha}(t))_{t\in [0,\tau)}$ is a closed MLO in $X,$ provided that $(S_{\alpha}(t))_{t\in [0,\tau)}$ is locally equicontinuous.
Furthermore, $\hat{{\mathcal A}}\subseteq C^{-1}\hat{{\mathcal A}}C$ in the MLO sense, with the equality in the case that
the operator $C$ is injective.

By a subgenerator of $(S_{\alpha}(t))_{t\in [0,\tau)}$ we mean any MLO ${\mathcal A}$ in $X$ satisfying the following two conditions:
\begin{itemize}
\item[(A)] $S_{\alpha}(t)x-g_{\alpha+1} (t)Cx=\int_0^tS_{\alpha}(s)y\,ds,\mbox{ whenever }t\in [0,\tau)\mbox{ and }y\in {\mathcal A}x.$
\item[(B)] For all $x\in X$ and $t\in [0,\tau),$  we have $\int^{t}_{0}S_{\alpha}(s)x\, ds \in D({\mathcal A})$ and
$S_{\alpha}(t)x-g_{\alpha+1}(t)Cx\in {\mathcal A}\int_0^t S_{\alpha}(s)x\,ds.$
\end{itemize}
If $(S_{\alpha}^{1}(t))_{t\in [0,\tau)}\subseteq  L(X),$ resp. $(S_{\alpha}^{2}(t))_{t\in [0,\tau)}\subseteq  L(X),$ is strongly continuous and satisfies only (B), resp. (A), with $(S_{\alpha}(t))_{t\in [0,\tau)}$ replaced therein with $(S_{\alpha}^{1}(t))_{t\in [0,\tau)},$ resp. $(S_{\alpha}^{2}(t))_{t\in [0,\tau)}$, then we say that
$(S_{\alpha}^{1}(t))_{t\in [0,\tau)},$ resp. $(S_{\alpha}^{2}(t))_{t\in [0,\tau)},$ is an $\alpha$-times integrated $C$-existence family with a subgenerator ${\mathcal A},$ resp.,
$\alpha$-times integrated $C$-uniqueness family with a subgenerator ${\mathcal A}.$

The notion of an exponentially
equicontinuous, analytic $\alpha$-times integrated $C$-semigroup is introduced in the following definition.

\begin{defn}\label{jebai-ga-mloss}
\begin{itemize}
\item[(i)]
Let $\nu \in (0,\pi],$ and
let $(S_{\alpha}(t))_{t\geq 0}$ be an $\alpha$-times integrated $C$-semigroup. Then it is said that $(S_{\alpha}(t))_{t\geq 0}$ is an analytic
$\alpha$-times integrated $C$-semigroup of angle $\nu,$ if
there exists a function ${\bf S}_{\alpha} : \Sigma_{\nu} \rightarrow L(X)$
which satisfies that, for every $x\in X,$ the mapping $z\mapsto {\bf
S}_{\alpha}(z)x,$ $z\in \Sigma_{\nu}$ is analytic as well as that:
\begin{itemize}
\item[(a)] ${\bf S}_{\alpha}(t)=S_{\alpha}(t),\ t>0$ and
\item[(b)] $\lim_{z\rightarrow
0,z\in \Sigma _{\gamma }}{\bf S}_{\alpha}(z)x=\delta_{\alpha,0}Cx$ for
all $\gamma \in (0,\nu)$ and $x\in X.$
\end{itemize}
\item[(ii)] Let $(S_{\alpha}(t))_{t\geq 0}$ be an analytic $\alpha$-times integrated $C$-semigroup of angle $\nu \in (0,\pi].$
Then it is said that $(S_{\alpha}(t))_{t\geq 0}$ is an exponentially
equicontinuous, analytic $\alpha$-times integrated $C$-semigroup of angle $\nu,$
resp. equicontinuous analytic $\alpha$-times integrated $C$-semigroup of angle $\nu,$ if for every
$\gamma\in(0,\nu),$ there exists $ \omega_{\gamma} \geq 0,$ resp.
$\omega_{\gamma}=0,$ such that the family $\{e^{-\omega_{\gamma}\Re
z}{\bf S}_{\alpha}(z) : z\in \Sigma_{\gamma}\}\subseteq L(X)$ is equicontinuous.
\end{itemize}
\end{defn}

For more details about degenerate fractionally integrated $C$-semigroups, we refer the reader to \cite{catania}.

\section{Definition and main properties of complex powers of operators satisfying the condition (H)$_{0}$}

Assume that
the condition (H)$_{0},$ resp. (HS)$_{0},$ holds.
Without loss of generality, we may assume that there exists a number $\lambda_{0} \in \text{int}(\rho_{C}({\mathcal A})) \setminus (P_{\alpha, \varepsilon, c} \cup B_{d}),$ resp., $\lambda_{0} \in \text{int}(\rho_{C}({\mathcal A})) \setminus (\Sigma_{\vartheta} \cup B_{d}).$ Then we can prove inductively (cf. also Theorem \ref{gen-res-MLO}(i)) that, for
every $z\in \rho_{C}({\mathcal A}) \setminus \{\lambda_{0}\}:$
\begin{equation}\label{resequ-rundes}
\bigl(z-{\mathcal A}\bigr)^{-1}C\bigl(\lambda_{0}-{\mathcal A}\bigr)^{-k}Cx
=\frac{(-1)^{k}}{(z-\lambda_{0})^{k}}\bigl(z-{\mathcal A}\bigr)^{-1}C^{2}x +\sum
\limits^{k}_{i=1}\frac{(-1)^{k-i}\bigl(\lambda_{0}-{\mathcal A}\bigr)^{-i}C^{2}x}{\bigl(z-\lambda_{0}\bigr)^{k+1-i}}.
\end{equation}
Strictly speaking, for $k=1$ this is the usual resolvent equation. Suppose that (\ref{resequ-rundes}) holds for all natural numbers $\leq k.$ Then
(\ref{creso}) shows that
\begin{align*}
\bigl(z&-{\mathcal A}\bigr)^{-1}C\bigl(\lambda_{0}-{\mathcal A}\bigr)^{-(k+1)}Cx
\\ & =\bigl(z-{\mathcal A}\bigr)^{-1}C\frac{(-1)^{k}}{k!}\Biggl(\frac{d^{k}}{d\lambda^{k}}
\bigl(\lambda-{\mathcal A}\bigr)^{-1}Cx\Biggr)_{\lambda =\lambda_{0}}
\\ & =\bigl(z-{\mathcal A}\bigr)^{-1}C\frac{(-1)^{k}}{k!}\Biggl(\frac{d}{d\lambda}\Biggl[(-1)^{k-1}(k-1)!
\bigl(\lambda-{\mathcal A}\bigr)^{-k}Cx\Biggr]\Biggr)_{\lambda =\lambda_{0}}
\\ & =\frac{(-1)}{k}\Biggl(  \frac{d}{d\lambda}\Biggl[ \bigl(z-{\mathcal A}\bigr)^{-1}C\bigl(\lambda-{\mathcal A}\bigr)^{-k}Cx     \Biggr] \Biggr)_{\lambda=\lambda_{0}}
\end{align*}
and we can employ the inductive hypothesis and a simple computation involving the equality
\begin{align*}
\Biggl(\frac{d}{d\lambda} \Bigl[ \bigl(\lambda-{\mathcal A}\bigr)^{-i}C^{2}x\Bigr] \Biggr)_{\lambda=\lambda_{0}}
=(-i) \bigl(\lambda-{\mathcal A}\bigr)^{-i-1}C^{2}x,\quad x\in X,\ i \in {\mathbb N}_{0}
\end{align*}
in order to see that (\ref{resequ-rundes}) holds with $k$ replaced with $k+1$ therein.
Set
\begin{align}\label{marseljeza}
C_{1}:=C\bigl(\lambda_{0} -{\mathcal A}\bigr)^{-\lfloor \alpha+2 \rfloor}C ,\mbox{ if }\alpha>-1,\mbox{ and }
C_{1}:=C,\mbox{ if }\alpha=-1.
\end{align}
Then Theorem \ref{C-favini}(i) implies by iteration that $C_{1}$ commutes with ${\mathcal A}.$ Furthermore, the validity of (H)$_{0}$, resp. (HS)$_{0}$, implies by Theorem \ref{gen-res-MLO}(i) that the following holds:
\begin{itemize}
\item[(H):] There exist real numbers $d\in (0,1],$ $c\in (0,1)$ and $\varepsilon \in (0,1]$ such
that $P_{\alpha, \varepsilon, c} \cup B_{d} \subseteq \rho_{C_{1}}({\mathcal A}),$ the operator family $\{ (1+|\lambda|)^{-1}(\lambda - {\mathcal A})^{-1}C_{1} : \lambda \in P_{\alpha, \varepsilon, c} \cup B_{d}\} \subseteq L(X)$ is equicontinuous,
the mapping $\lambda \mapsto (\lambda - {\mathcal A})^{-1}C_{1}$ is strongly analytic on int$(P_{\alpha, \varepsilon, c} \cup B_{d})$ and strongly continuous on $\partial (P_{\alpha, \varepsilon, c} \cup B_{d}),$
\end{itemize}
resp.,
\begin{itemize}
\item[(HS):]
There exist real numbers $d\in (0,1]$ and $\vartheta \in (0,\pi/2)$ such
that $\Sigma_{\vartheta} \cup B_{d} \subseteq \rho_{C_{1}}({\mathcal A}),$ the operator family $\{ (1+|\lambda|)^{-1}(\lambda - {\mathcal A})^{-1}C_{1} : \lambda \in \Sigma_{\vartheta} \cup B_{d}\} \subseteq L(X)$ is equicontinuous,
the mapping $\lambda \mapsto (\lambda - {\mathcal A})^{-1}C_{1}$ is strongly analytic on int$(\Sigma_{\vartheta}\cup B_{d})$ and strongly continuous on $\partial (\Sigma_{\vartheta} \cup B_{d}).$
\end{itemize}

So, the condition (H), resp. (HS), holds and $C_{1}{\mathcal A} \subseteq {\mathcal A}C_{1};$ in particular,
$-{\mathcal A}$ is $C_{1}$-positive, resp. $-{\mathcal A},$ is $C_{1}$-sectorial of angle $\pi-\vartheta$ and $B_{d}\subseteq \rho_{C_{1}}({-\mathcal A}).$

Put $\Gamma_{1}(\alpha, \varepsilon, c,d):=\{\xi+i\eta : \xi \leq
-\varepsilon,\ \eta=-c(1+|\xi|)^{-\alpha}\},$ $\Gamma_{2}(\alpha,
\varepsilon, c,d):=\{\xi+i\eta : \xi^{2}+\eta^{2}=d^{2},\ \xi\geq
- \varepsilon\}$ and $\Gamma_{3}(\alpha, \varepsilon, c,d):=\{\xi+i\eta
: \xi \leq -\varepsilon,\ \eta=c(1+|\xi|)^{-\alpha}\}.$ The curve
$\Gamma(\alpha, \varepsilon, c,d):= \Gamma_{1}(\alpha, \varepsilon, c,d)
\cup \Gamma_{2}(\alpha, \varepsilon, c,d) \cup \Gamma_{3}(\alpha,
\varepsilon, c,d)$ is oriented so that $\hbox{Im}(\lambda)$ increases
along $\Gamma_{2}(\alpha, \varepsilon, c,d)$ and that $\hbox{Im}(\lambda)$ decreases
along $\Gamma_{1}(\alpha, \varepsilon, c,d)$ and $\Gamma_{3}(\alpha, \varepsilon, c,d).$
Since there is no risk
for confusion, we also write $\Gamma$ for $\Gamma(\alpha,
\varepsilon, c,d).$ We similarly define the curves $\Gamma_{1,S}(\vartheta, d),$ $\Gamma_{2,S}(\vartheta, d),$ $\Gamma_{3,S}(\vartheta, d)$ and $\Gamma_{S}(\vartheta, d)$ for $\vartheta \in (0,\pi/2)$ and $d\in (0,1].$

Define
\begin{align*}
f_{C_{1}} ({\mathcal A}) x := \frac{1}{2 \pi i} \int_{\Gamma} f (z)
\bigl(z + {\mathcal A}\bigr)^{-1} C_{1} x \, d z, \quad x \in X,
\end{align*}
where $f(z)$ is a
holomorphic function on an open neighborhood $\Omega_{\alpha, \varepsilon, c,d}$ of $-(P_{\alpha, \varepsilon, c} \cup B_{d}) \setminus (-\infty,0]$
and the estimate
\begin{eqnarray*}
\begin{split}
| f (z) | \leq M |z|^{-s}, \quad z \in \Omega_{\alpha, \varepsilon, c,d}
\end{split}
\end{eqnarray*}
holds for some positive number $s>0.$ Denote by ${\mathcal H}$ the class consisting of such functions.
Then an application of Cauchy's theorem
shows that the definition of $f_{C_{1}} ({\mathcal A}) $ does not depend on a particular choice of curve
$\Gamma (\alpha, \varepsilon, c,d)$ (with the meaning clear). Furthermore, a standard calculus involving the Cauchy theorem, the Fubini theorem and 
Theorem \ref{C-favini}(ii) shows that
\begin{eqnarray}\label{e:homomorphism}
\begin{split}
f_{C_{1}} ({\mathcal A}) g_{C_{1}} ({\mathcal A}) = \bigl(f g\bigr)_{C_{1}} ({\mathcal A}) C_{1}, \quad f,\ g,\ f g \in
{\mathcal H}.
\end{split}
\end{eqnarray}

Given $b\in {\mathbb C}$ with $\Re b>0,$ set $(-{\mathcal A})_{C_{1}}^{-b} := (z^{-b})_{C_{1}} ({\mathcal A})$ and
$(-{\mathcal A})_{C_{1}}^{-0} := C_{1}$. By Remark \ref{univ-prof}(ii) and the residue theorem, we get $(-{\mathcal A})_{C_{1}}^{-n} = (-{\mathcal A})^{-n} C_{1}$ ($n \in
\mathbb{N}$); moreover, $(-{\mathcal A})_{C_{1}}^{-b}C_{1}=C_{1}(-{\mathcal A})_{C_{1}}^{-b}$ ($\Re b> 0$), the mapping
$b\mapsto (-{\mathcal A})_{C_{1}}^{-b}x,$ $\Re b>0$ is analytic for every fixed $x\in
X,$ and the following holds:
$$
\frac{d}{db}(-{\mathcal A})_{C_{1}}^{-b}x=\frac{(-1)}{2\pi i}\int_{\Gamma}
(\ln z) z^{-b}\bigl(z + {\mathcal A}\bigr)^{-1} C_{1} x \, d z, \quad x \in X,\ \Re
b>0.
$$
Applying the equality (\ref{e:homomorphism}) once more, we get that
$$
(-{\mathcal A})_{C_{1}}^{-b_1} (-{\mathcal A})_{C_{1}}^{-b_2} = (-{\mathcal A})_{C_{1}}^{- (b_1 + b_2)} C_{1}, \quad \Re b_1,\  \Re
b_2 > 0.
$$
It is very simple to prove that
\begin{align*} %\label{e:real integral for Ac -b}
(-{\mathcal A})_{C_{1}}^{-b}x
= & -\frac{\sin \pi b}{\pi} \int_0^\infty \lambda^{-b} \bigl(\lambda
- {\mathcal A}\bigr)^{-1} C_{1} x \, d \lambda, \quad 0 < \Re b < 1,\quad x\in X,
\end{align*}
so that the family $\{ (-{\mathcal A})_{C_{1}}^{-b} :
0<b<1 \}$ is equicontinuous.
Define now the powers with negative imaginary part of exponent
\index{fractional powers!with negative imaginary part of exponent}
by
$$
\bigl(-{\mathcal A}\bigr)_{-b} := C_{1}^{-1} \bigl(-{\mathcal A}\bigr)_{C_{1}}^{-b}, \quad \Re b > 0.
$$
Then $(-{\mathcal A})_{-b}$ is a closed MLO and $(-{\mathcal A})_{-n} = C_{1}^{-1} (-{\mathcal A})^{-n} C_{1}$ ($n \in
\mathbb{N}$). We define the
powers with positive imaginary part of exponent\index{fractional powers!with positive imaginary part of exponent} by
$$
\bigl(-{\mathcal A}\bigr)_{b} := \bigl((-{\mathcal A})_{-b}\bigr)^{-1} = \bigl((-{\mathcal A})_{C_{1}}^{-b}\bigr)^{-1} C_{1},
\quad \Re b > 0.
$$
Clearly, $(-{\mathcal A})_{n} = C_{1}^{-1} (-{\mathcal A})^{n} C_{1}$ for every $n \in \mathbb{N}$, and
$(-{\mathcal A})_b$ is a closed MLO due to the fact that $(-{\mathcal A})_{-b}$ is a closed MLO ($b\in {\mathbb C}_{+}$).
Following \cite[Definition 7.1.2]{martinez} and our previous analyses of non-degenerate case \cite{TAJW}, we introduce the purely
imaginary powers\index{fractional powers!purely
imaginary} of $-{\mathcal A}$ as follows: Let $\tau \in {\mathbb R} \setminus \{0\}.$
Then the power $(-{\mathcal A})_{i\tau}$ is defined by
$$
\bigl(-{\mathcal A}\bigr)_{i\tau}:=C_{1}^{-2}\bigl(1-{\mathcal A}\bigr)_{2}(-{\mathcal A})_{-1}(-{\mathcal A})_{1+i\tau}\bigl(1-{\mathcal A}\bigr)_{-2}C_{1}^{2},
$$
where $(1-{\mathcal A})_{2}=C_{1}^{-1}(1-{\mathcal A})^{2}C_{1}$ and $(1-{\mathcal A})_{-2}=C_{1}^{-1}(1-{\mathcal A})^{-2}C_{1}.$
We will later see (cf. (S.4)) that $C_{1}(D({\mathcal A}^{2}))\subseteq D((-{\mathcal A})_{1+i\tau}),$ so that the closedness of
$(-{\mathcal A})_{i\tau}$ follows from a simple computation involving the closedness of $(-{\mathcal A})_{1+i\tau}.$

Further on, the Cauchy integral formula and (\ref{creso}) implies that the operator family $\{\lambda^{k}(\lambda-{\mathcal A})^{-k}C_{1} : \lambda >0\}\subseteq L(X)$ is equicontinuous for all $k\in {\mathbb N}.$ If $y\in {\mathcal A}^{k}x$ for some $k\in {\mathbb N}$ and $x\in D({\mathcal A}^{k}),$ then there exists a sequence $(y_{j})_{1\leq j\leq k}$ in  $X$ such that $y_{k}=y$ and $(x,y_{1})\in {\mathcal A},$ $(y_{1},y_{2})\in {\mathcal A},\cdot \cdot \cdot ,$ $(y_{k-1},y_{k})\in {\mathcal A}.$ Then we can inductively prove with the help of Theorem \ref{C-favini}(i) that
\begin{align*}
\lambda^{k}\bigl(\lambda-{\mathcal A}\bigr)^{-k}C_{1}x=C_{1}x+\sum \limits_{j=1}^{k}\binom{k}{j}\bigl(\lambda-{\mathcal A}\bigr)^{-j}
C_{1}y_{j},\quad \lambda>0,
\end{align*}
which implies that $\lim_{\lambda \rightarrow +\infty}\lambda^{k}\bigl(\lambda-{\mathcal A}\bigr)^{-k}C_{1}x=C_{1}x,$  $k\in {\mathbb N},$ $x\in D({\mathcal A}^{k});$
 cf. \cite[Lemma 2.7]{TAJW}. The assertion of \cite[Lemma 2.5]{TAJW} also holds in our framework.

Now we will reconsider multivalued analogues of some statements established in \cite[Theorem 2.8, Theorem 2.10, Lemma 2.14]{TAJW}.
\begin{itemize}
\item[(S.1)] Suppose $\Re b\neq 0.$ Then it is checked at once that $(-{\mathcal A})_{b} \subseteq C_{1}^{-1}(-{\mathcal A})_{b} C_{1},$
with the equality in the case that the operator $C_{1}$ is injective.
\item[(S.2)] Suppose $\Re b_{1} <0$ and $\Re b_{2} <0.$ Then $R(C_{1})\subseteq D((-{\mathcal A})_{z}),$ $\Re z<0,$
\begin{align*}
\bigl(-{\mathcal A}\bigr)_{C_{1}}^{b_{1}+b_{2}}x& \in C_{1}^{-1}\bigl(-{\mathcal A}\bigr)_{C_{1}}^{b_{1}+b_{2}}C_{1}x
=C_{1}^{-1}\bigl(-{\mathcal A}\bigr)_{C_{1}}^{b_{1}}\bigl(-{\mathcal A}\bigr)_{C_{1}}^{b_{2}}x\\& \subseteq
C_{1}^{-1}\bigl(-{\mathcal A}\bigr)_{C_{1}}^{b_{1}}C_{1}^{-1}\bigl(-{\mathcal A}\bigr)_{C_{1}}^{b_{2}}C_{1}x=\bigl(-{\mathcal A}\bigr)_{b_{1}}\bigl(-{\mathcal A}\bigr)_{b_{2}}C_{1}x,\quad x\in X.
\end{align*}
This, in turn, implies $(-{\mathcal A})_{C_{1}}^{b_{1}+b_{2}}\subseteq (-{\mathcal A})_{b_{1}}(-{\mathcal A})_{b_{2}}C_{1},$
$C_{1}^{-1}(-{\mathcal A})_{C_{1}}^{b_{1}+b_{2}}\subseteq C_{1}^{-1}(-{\mathcal A})_{b_{1}}(-{\mathcal A})_{b_{2}}C_{1}$ and
\begin{align}\label{seksi-govno}
\bigl(-{\mathcal A}\bigr)_{b_{1}+b_{2}} \subseteq C_{1}^{-1}\bigl(-{\mathcal A}\bigr)_{b_{1}}\bigl(-{\mathcal A}\bigr)_{b_{2}}C_{1}.
\end{align}
Let $y\in (-{\mathcal A})_{b_{1}}(-{\mathcal A})_{b_{2}}x.$ Thus, $C_{1}y\in
 (-{\mathcal A})_{C_{1}}^{b_{1}}C_{1}^{-1}(-{\mathcal A})_{C_{1}}^{b_{2}}x.$ This yields the existence of an element
$u\in C_{1}^{-1}(-{\mathcal A})_{C_{1}}^{b_{2}}x$ such that $C_{1}z=(-{\mathcal A})_{C_{1}}^{b_{2}}x$ and $C_{1}y=(-{\mathcal A})_{C_{1}}^{b_{1}}u.$ So, $C_{1}^{2}y=(-{\mathcal A})_{C_{1}}^{b_{1}}C_{1}u=(-{\mathcal A})_{C_{1}}^{b_{1}}(-{\mathcal A})_{C_{1}}^{b_{2}}C_{1}x,$ $C_{1}y\in C_{1}^{-1}(-{\mathcal A})_{C_{1}}^{b_{1}+b_{2}}C_{1}x=(-{\mathcal A})_{b_{1}+b_{2}}C_{1}x
$ and $y\in C_{1}^{-1}(-{\mathcal A})_{b_{1}+b_{2}}C_{1}x.$ Hence,
\begin{align}\label{seksi-govno-franc}
\bigl(-{\mathcal A}\bigr)_{b_{1}}\bigl(-{\mathcal A}\bigr)_{b_{2}} \subseteq C_{1}^{-1}\bigl(-{\mathcal A}\bigr)_{b_{1}+b_{2}} C_{1}.
\end{align}
\item[(S.3)]
Suppose now that $\Re b_{1} >0$ and $\Re b_{2} >0.$ Using the equations (\ref{seksi-govno})-(\ref{seksi-govno-franc}) with $b_{1}$ and $b_{2}$ replaced respectively by $-b_{1}$ and $-b_{2}$ therein, and taking the inverses after that, it readily follows from (S.2)  that (\ref{seksi-govno})-(\ref{seksi-govno-franc}) holds in this case.
\item[(S.4)] Repeating almost literally the arguments from the proof
of \cite[Theorem 2.8(ii.2)]{TAJW}, we can deduce the following: Suppose that $\Re b>0$ and $k =\lceil \Re b \rceil,$ resp. $k=\lceil \Re b \rceil+1,$ provided that $\Re b \notin {\mathbb N},$
resp. $\Re b \in {\mathbb N}.$ Let $x=C_{1}y$ for some $y\in D({\mathcal A}^{k}).$ Then there exists a sequence $(y_{j})_{1\leq j \leq k}$ in $X$
such that $(y,y_{1})\in {\mathcal A},$ $(y_{1},y_{2})\in {\mathcal A}, \cdot \cdot \cdot,$ $(y_{k-1},y_{k})\in {\mathcal A}.$ Furthermore,
$C_{1}(D({\mathcal A}^{k}))\subseteq D((-{\mathcal A})_{b})$ and, for every such a sequence, we have
\begin{align*}
\frac{1}{2\pi i}\int
\limits_{\Gamma}z^{b-\lfloor \Re b\rfloor-1}\bigl(z +{\mathcal A}
\bigr)^{-1}C_{1}y_{k}\,dz \in \bigl(-{\mathcal A}\bigr)_{b}x.
\end{align*}
\item[(S.5)] The assertion of \cite[Theorem 2.8(iii)]{TAJW} is not really interested in multivalued case because $(-{\mathcal A})_{b}x$ is not singleton, in general.
\item[(S.6)] Let $\tau \in {\mathbb R}.$ Then a straightforward computation involving (S.1) shows that $(-{\mathcal A})_{i\tau} \subseteq C_{1}^{-1}(-{\mathcal A})_{i\tau} C_{1}.$
The equality $(-{\mathcal A})_{i\tau} =C_{1}^{-1}(-{\mathcal A})_{i\tau} C_{1}$ can be also trivially verified provided
that the operator $C_{1}$ is injective.
\item[(S.7)] Let $x=C_{1}y$ for some $y\in D({\mathcal A}),$ and let $\tau \in {\mathbb R}.$ Keeping in mind (\ref{resequ-rundes}), (\ref{e:homomorphism}), (S.4), the residue theorem and Theorem \ref{integracija-tricky}, we can prove as in single-valued linear case that:
$$
\frac{1}{2\pi
i}\int
\limits_{\Gamma}z^{-1+i\tau}\frac{z}{z+1}\bigl(z+{\mathcal A}\bigr)^{-1}C_{1}^{2}x\, dz
\in (1-{\mathcal A})(-{\mathcal A})_{-1}(-{\mathcal A})_{1+i\tau}\bigl(1-{\mathcal A}\bigr)_{-2}C_{1}^{2}x.
$$
Let $u\in (1-{\mathcal A})y.$
Using Lemma \ref{integracija-tricky} and Theorem \ref{C-favini}(i), we get from the above that
\begin{align*}
C_{1}^{-3}(1-{\mathcal A}) & C_{1}\Biggl[ \frac{1}{2\pi
i}\int
\limits_{\Gamma}z^{-1+i\tau}\frac{z}{z+1}\bigl(z+{\mathcal A}\bigr)^{-1}C_{1}^{2}x\, dz \Biggr]
\\ & =\frac{1}{2\pi
i}\int
\limits_{\Gamma}z^{-1+i\tau}\frac{z}{z+1}\bigl(z+{\mathcal A}\bigr)^{-1}C_{1}u\, dz \in (-{\mathcal A})_{i\tau}x,
\end{align*}
so that $C_{1}(D({\mathcal A}))\subseteq D((-{\mathcal A})_{i\tau}).$ Unfortunately, a great number of important properties of purely imaginary powers established in \cite[Theorem 2.10]{TAJW} does not continue to hold in multivalued linear case.
\item[(S.8)] Let $n\in {{\mathbb N}_{0}},$ let $b\in {\mathbb C}$ and let $\Re
b\in (0,n+1) \setminus {\mathbb N}.$ Set $(1-b)(2-b)\cdot \cdot \cdot (n-b):=1$ for $n=0.$ Then, for every $x\in X,$ we have
$$
C_{1}^{n}(-{\mathcal A})_{C_{1}}^{-b}x=\frac{(-1)^{n}n!}{(1-b)(2-b)\cdot \cdot \cdot
(n-b)}\frac{\sin \pi (n-b)}{\pi}\int
\limits^{\infty}_{0}t^{n-b}\bigl(t-{\mathcal A}\bigr)^{-(n+1)}C_{1}^{n+1}x\,dt.
$$
This can be shown following the lines of the
proof of \cite[Theorem 5.27, p. 138]{eng}.
\end{itemize}

\section{The existence and uniqueness of solutions of abstract incomplete differential inclusions}

In this section, we assume that the condition (HS)$_{0}$ holds. Define $C_{1}$ through (\ref{marseljeza}). Then (HS) holds
and we can define the fractional powers of $-{\mathcal A}$ as it has been done in the third section of paper.

Following A. V. Balakrishnan \cite{balak}, define
\begin{align*}
f_{t}(\lambda)& :=\frac{1}{\pi}e^{-t \lambda^{\gamma}\cos \pi
\gamma}\sin \bigl( t\lambda^{\gamma}\sin \pi \gamma
\bigr)\\ &=\frac{1}{2\pi i}\Bigl( e^{-t
\lambda^{\gamma}e^{-i \pi \gamma }}-e^{-t \lambda^{\gamma}e^{i \pi
\gamma }}\Bigr),\quad t>0,\ \lambda>0.
\end{align*}
This function enjoys the following properties:
\begin{itemize}
\item[Q1.] $|f_{t}(\lambda)|\leq
\pi^{-1}e^{-\lambda^{\gamma}\epsilon_{t}},$ $\lambda>0,$ where
$\epsilon_{t}:=t\cos \pi \gamma >0.$
\item[Q2.] $|f_{t}(\lambda)|\leq
\gamma t\lambda^{\gamma}e^{-t \lambda^{\gamma}\sin \epsilon_{t}},$
$\lambda>0.$
\item[Q3.] $\int^{\infty}_{0}\lambda^{n}f_{t}(\lambda)\,
d\lambda=0,\ n\in {{\mathbb N}_{0}},\ t>0.$
\item[Q4.] Let $m\geq -1.$ Then Q1./Q2. together imply that the improper integral
$\int^{\infty}_{0}\lambda^{n}f_{t}(\lambda)( \lambda -{\mathcal A})^{-1}C_{1}\cdot
\, d\lambda$ is absolutely convergent and defines a bounded linear
operator on $ X$ ($n\in {{\mathbb N}_{0}}$).
\end{itemize}
Put now, for $0<\gamma<1/2$,
\begin{align*}
S_{\gamma}(t)x:=\int \limits^{\infty}_{0}f_{t}(\lambda) \bigl(
\lambda - {\mathcal A} \bigr)^{-1}C_{1}x\, d\lambda,\quad t>0,\ x\in X.
\end{align*}
Then $S_{\gamma}(t)\in L(X),$ $t>0$ and the following holds:

\begin{lem}\label{funk-racun-mlo}
We have
\begin{align}\label{djezine}
S_{\gamma}(t)=\bigl (e^{-tz^{\gamma}}\bigr)_{C_{1}}({\mathcal A}),\quad t>0,\ 0<\gamma<1/2.
\end{align}
Furthermore, $
S_{\gamma}(t)$ can be defined by \emph{(\ref{djezine})} for all $t  \in \Sigma_{(\pi /2)-\gamma \pi},$ and the mapping
$t\mapsto S_{\gamma}(t),$ $ t  \in \Sigma_{(\pi /2)-\gamma \pi}$ is strongly analytic ($0<\gamma<1/2$).
\end{lem}

\begin{proof}
Observe that, for every $ t=t_{1}+it_{2}  \in \Sigma_{(\pi /2)-\gamma \pi}$ and $ z \in {\mathbb C} \setminus \{0\}, $
we have
\begin{align*}
\Bigl| e^{-tz^{\gamma}}\Bigr | \leq e^{-|z|^{\gamma}t_{1}\cos (\gamma \arg (z))[1-|\tan (\arg (t))|\tan (\gamma \arg (z))]}.
\end{align*}
Keeping this estimate in mind, it is very simple to deform the path of integration $\Gamma_{S}(\vartheta, d)$ into the negative real axis, showing that for each $t  \in \Sigma_{(\pi /2)-\gamma \pi}$ and $ x \in X$ we have:
\begin{align*}
\frac{1}{2 \pi i} \int_{\Gamma_{S}(\vartheta, d)} & e^{-t\lambda^{\gamma}}
\bigl(\lambda + {\mathcal A}\bigr)^{-1} C_{1} x \, d z
\\ & =\frac{1}{2\pi i}\int \limits^{\infty}_{0}\Bigl( e^{-t
\lambda^{\gamma}e^{-i \pi \gamma }}-e^{-t \lambda^{\gamma}e^{i \pi
\gamma }}\Bigr)\bigl(
\lambda - {\mathcal A} \bigr)^{-1}C_{1}x\, d\lambda .
\end{align*}
The remaining part of proof is left to the interested reader.
\end{proof}

Set
\begin{align*}
\varphi_{\gamma}:=(\pi /2)-\gamma (\pi-\vartheta),\ \mbox{ for }0< \gamma \leq 1/2.
\end{align*}

\begin{thm}\label{thm-mlo-fgrp}
Put
$S_{\gamma}(0):=C_{1},$
$S_{\gamma,\zeta}(t)x:=\int^{t}_{0}g_{\zeta}(t-s)S_{\gamma}(s)x\,
ds,$ $x\in X,$ $t\in \Sigma_{(\pi /2)-\gamma \pi}$ ($\zeta>0$), and
$S_{\gamma,0}(t):=S_{\gamma}(t),$ $t\in \Sigma_{(\pi /2)-\gamma \pi}.$ Then the family
$\{S_{\gamma}(t) : t>0\}$ is
equicontinuous, and there exist strongly analytic operator families $({\bf S}_{\gamma}(t))_{t\in \Sigma_{\varphi_{\gamma}}}$ and
$({\bf S}_{\gamma,\zeta}(t))_{t\in \Sigma_{\varphi_{\gamma}}}$ such that
${\bf S}_{\gamma}(t)=S_{\gamma}(t),$ $ t>0$
and ${\bf S}_{\gamma,\zeta}(t)=S_{\gamma,\zeta}(t),$ $t>0.$
Furthermore, the following holds:
\begin{itemize}
\item[(i)] ${\bf S}_{\gamma}(t_{1}){\bf S}_{\gamma}(t_{2})=
{\bf S}_{\gamma}(t_{1}+t_{2})C_{1}$ for all $t_{1},\ t_{2}\in \Sigma_{\varphi_{\gamma}}.$
\item[(ii)] We have
$\lim_{t\rightarrow 0,t\in \Sigma_{\varphi_{\gamma}-\epsilon}}{\bf S}_{\gamma}(t)x=C_{1}x,$ $x\in \overline{D({\mathcal A})},$ $\epsilon
\in (0,\varphi_{\gamma}).$
\item[(iii)]
${\bf S}_{\gamma}(z)(-{\mathcal A})_{\nu}\subseteq
(-{\mathcal A})_{\nu}{\bf S}_{\gamma}(z),$ $z\in \Sigma_{\varphi_{\gamma}},$
$\nu \in {{\mathbb C}_{+}}.$
\item[(iv)] If $D({\mathcal A})$ is dense in $X,$ then $(S_{\gamma}(t))_{t\geq 0}$ is an equicontinuous analytic
$C_{1}$-regularized semigroup of angle $\varphi_{\gamma}.$ Moreover, $(S_{\gamma}(t))_{t\geq 0}$ is a
$C_{1}$-regularized existence family
with a subgenerator
$-(-{\mathcal A})_{\gamma}$
and the supposition $(x,y) \in -(-{\mathcal A})_{\gamma}$ implies $(C_{1}x,C_{1}y) \in \hat{{\mathcal A}}_{\gamma},$ where
$
\hat{{\mathcal A}}_{\gamma}$ is the integral generator of $(S_{\gamma}(t))_{t\geq 0};$
otherwise, for every $\zeta>0,$ $(S_{\gamma,\zeta}(t))_{t\geq 0}$ is an exponentially equicontinuous, analytic $\zeta$-times integrated
$C_{1}$-regularized semigroup,\\ $(S_{\gamma,\zeta}(t))_{t\geq 0}$ is a $\zeta$-times integrated $C_{1}$-existence family
with a subgenerator
$-(-{\mathcal A})_{\gamma}$ and  the supposition $(x,y) \in -(-{\mathcal A})_{\gamma}$ implies $(C_{1}x,C_{1}y) \in \hat{{\mathcal A}}_{\gamma}.$
\item[(v)] For every $x\in X,\
t\in \Sigma_{(\pi /2)-\gamma \pi}$ and $n\in {\mathbb N},$ we have
\begin{align}\label{qwer-zorbinjo}
\Biggl(S_{\gamma}(t)x,-\int^{\infty}_{0}\lambda^{n}
f_{t}(\lambda)\bigl(\lambda-{\mathcal A}\bigr)^{-1}C_{1}x\, d\lambda\Biggr)\in {\mathcal A}^{n}.
\end{align}
\item[(vi)] Suppose $\beta>0.$
Denote by $\Omega_{\theta,\gamma},$ resp. $\Psi_{\gamma},$
% u slucaju $m=-1,$ angularni set sadrzi zatvorenje od D(A)!!!
 the continuity set of $(S_{\gamma}(te^{i\theta}))_{t>0},$ resp.
$(S_{\gamma}(t))_{t\in \Sigma_{\varphi_{\gamma}}}.$ Then, for
every $x\in \Omega_{\theta,\gamma},$ the incomplete abstract Cauchy
inclusion\index{incomplete abstract Cauchy
inclusion} \index{fractional derivatives!Liouville right-sided}
\[ (FP_{\beta}) : \left\{
\begin{array}{l}
u\in C^{\infty}\bigl((0,\infty ): X\bigr), \\
D^{\beta}_{-}u(t)\in e^{i\theta \beta} (-{\mathcal A})_{\gamma \beta} u(t),\ t>0, \\
\lim_{t\rightarrow 0+}u(t)=C_{1}x,\\
\mbox{the set } \{ u(t) : t>0\} \mbox{ is bounded in }X,
\end{array}%
\right.
\]
has a solution $u(t)=S_{\gamma}(te^{i\theta})x,$ $t>0,$ which can be
analytically extended to the sector $\Sigma_{\varphi_{\gamma}-|\theta|}.$ If, additionally, $x\in \Psi_{\gamma},$ then for
every $\delta \in (0,\varphi_{\gamma})$ and $j\in
{\mathbb{N}}_{0},$ we have that the set $\{z^{j}u^{(j)}(z) : z\in
\Sigma_{\delta}\} $ is bounded in $X.$
\end{itemize}
\end{thm}

\begin{proof}
The proof of (i) for real parameters $t_{1},\ t_{2}>0$ follows almost directly from definition of $S_{\gamma}(\cdot),$
by applying (\ref{e:homomorphism}); (v) is an easy consequence of Lemma \ref{integracija-tricky}, Theorem \ref{C-favini}(i)
and the property Q3. A very simple proof of (iii) is omitted. Set, for $|\theta|< \vartheta$ and $0<\gamma<1/2$,
$$
S_{\theta,\gamma}(t)x:=\int
\limits_{0}^{\infty}f_{t,\gamma}(\lambda) \bigl(
\lambda -e^{i\theta}{\mathcal A}\bigr)^{-1} C_{1}x\, d\lambda,\quad x\in X,\ t\in
\Sigma_{(\pi/2)-\gamma \pi}.
$$
Let $\theta_{1}\in (0,\vartheta)$ and $\theta_{2}\in (-\vartheta,0).$
Define
$$
{\bf S}_{\gamma}(t)x:=\left\{
\begin{array}{l}
S_{\gamma}(t)x,\ t\in \Sigma_{(\pi/2)-\gamma \pi}, \\
S_{\theta_{1},\gamma}(te^{-i\gamma \theta_{1}}),\mbox{ if }t\in
e^{i\gamma \theta_{1}}\Sigma_{(\pi/2)-\gamma \pi},\\
S_{\theta_{2},\gamma}(te^{-i\gamma \theta_{2}}),\mbox{ if }t\in
e^{i\gamma \theta_{2}}\Sigma_{(\pi/2)-\gamma \pi}.
\end{array}
\right.
$$
Then an elementary application of Cauchy formula shows that the operator family $({\bf S}_{\gamma}(t))_{t\in \Sigma_{\varphi_{\gamma}}}$ is well defined; furthermore, $({\bf S}_{\gamma}(t))_{t\in \Sigma_{\varphi_{\gamma}}}$ is strongly analytic
and equicontinuous on any proper subsector of $\Sigma_{\varphi_{\gamma}}$ (cf. also the proof of \cite[Theorem 2.9.48]{knjigah}). Using Theorem \ref{C-favini}(i), we get that
$\lim_{\lambda \rightarrow +\infty}[\lambda (\lambda-{\mathcal A})^{-1}C_{1}x-\lambda (\lambda+1)^{-1} C_{1}x]=0$ as $\lambda \rightarrow +\infty$ ($x\in D({\mathcal A})$). Taking into account this equality and the proof of \cite[Theorem 5.5.1(iv), p. 130]{martinez}, we get that $\lim_{t\rightarrow 0+}S_{\gamma}(t)x=C_{1}x,$ $x\in D({\mathcal A}).$ Now the remaining parts of proofs of (i)-(iii) can be straightforwardly completed.

We will prove (iv) provided that $D({\mathcal A})$ is dense in $X.$ It is clear that $(S_{\gamma}(t))_{t\geq 0}$ is an equicontinuous analytic
$C_{1}$-regularized semigroup $(S_{\gamma}(t))_{t\geq 0}$ of angle $\varphi_{\gamma}.$ Since, for every $t>0$ and $ x\in X,$
\begin{align*}
C_{1}\Bigl( -z^{-\gamma}e^{-tz^{\gamma}} +z^{-\gamma} \Bigr)_{C_{1}}({\mathcal A})x=-\bigl(z^{-\gamma} \bigr)_{C_{1}}({\mathcal A})\Bigl[ \bigl( e^{-tz^{\gamma}} \bigr)_{C_{1}}({\mathcal A})x-C_{1}x\Bigr],
\end{align*}
we have
\begin{align*}
C_{1}S_{\gamma,1}(t)x=-(-{\mathcal A})_{C_{1}}^{-\gamma}\Bigl[ S_{\gamma}(t)x -C_{1}x\Bigr],\quad t\geq 0,\ x\in X.
\end{align*}
This clearly implies that $(S_{\gamma,1}(t)x,S_{\gamma,1}(t)x -C_{1}x)\in -(-{\mathcal A})_{\gamma},$ $t\geq 0,$ $x\in X,$ so that
$(S_{\gamma,\zeta}(t))_{t\geq 0}$ is a $\zeta$-times integrated $C_{1}$-existence family
with a subgenerator
$-(-{\mathcal A})_{\gamma}.$ The supposition $(x,y) \in -(-{\mathcal A})_{\gamma}$ implies
$C_{1}x=-(-{\mathcal A})_{C_{1}}^{-\gamma}y$
and we can similarly prove that
$(C_{1}x,C_{1}y) \in \hat{{\mathcal A}}_{\gamma}.$

Arguing as in the proof of
\cite[Theorem 3.5(i)/(b)']{TAJW}, we get that, for every $x\in X$
and $t>0,$ the following equality holds, with $z=te^{i\theta}\in \Sigma_{(\pi /2)-\gamma \pi},$
\begin{equation}\label{metallica-MLOS}
D^{\beta}_{-}S_{\gamma}\bigl( te^{i\theta} \bigr)x=\frac{e^{i\theta
\beta}}{2\pi i} \int \limits^{\infty}_{0}\lambda^{\gamma
\beta}\Bigl[ e^{-i\gamma \beta \pi}e^{-z\lambda^{\gamma}e^{-i\pi
\gamma}}-e^{i\gamma \beta \pi}e^{-z\lambda^{\gamma}e^{i\pi \gamma}}
\Bigr]\bigl( \lambda -{\mathcal A}\bigr)^{-1}C_{1}x\, d\lambda.
\end{equation}
Deforming the path of integration $\Gamma_{S}(\vartheta, d)$ into the negative real axis, as it has been done in the proof of Lemma \ref{funk-racun-mlo}, we get
\begin{align}\label{metallica-MLOS-prim}
\bigl(\cdot^{\gamma \beta}e^{-z\cdot^{\gamma }}\bigr)_{C_{1}}({\mathcal A})=  \frac{1}{2\pi i} \int \limits^{\infty}_{0}\lambda^{\gamma
\beta}\Bigl[ e^{-i\gamma \beta \pi}e^{-z\lambda^{\gamma}e^{-i\pi
\gamma}}-e^{i\gamma \beta \pi}e^{-z\lambda^{\gamma}e^{i\pi \gamma}}
\Bigr]\bigl( \lambda -{\mathcal A}\bigr)^{-1}C_{1}x\, d\lambda.
\end{align}
Since
\begin{align*}
C_{1}\bigl(e^{-z\cdot^{\gamma }}\bigr)_{C_{1}}({\mathcal A})=\bigl(\cdot^{-\gamma \beta}\bigr)_{C_{1}}({\mathcal A})\bigl(\cdot^{-\gamma \beta}e^{-z\cdot^{\gamma }}\bigr)_{C_{1}}({\mathcal A}),
\end{align*}
(\ref{metallica-MLOS})-(\ref{metallica-MLOS-prim}) immediately implies that
\begin{align*}
\Bigl(e^{-i\theta \beta}D^{\beta}_{-}S_{\gamma}\bigl( te^{i\theta} \bigr)x,S_{\gamma}\bigl( te^{i\theta}\bigr)x\Bigr) \in C_{1}^{-1}(-{\mathcal A})_{C_{1}}^{\gamma \beta},
\quad t>0,\ x\in X,
\end{align*}
i.e.,
\begin{align*}
\Bigl(S_{\gamma}\bigl( te^{i\theta}\bigr)x,e^{-i\theta \beta}D^{\beta}_{-}S_{\gamma}\bigl( te^{i\theta} \bigr)x\Bigr) \in (-{\mathcal A})_{\gamma \beta},
\quad t>0,\ x\in X.
\end{align*}
The proof of (vi) now can be completed through a routine argument.
\end{proof}

\begin{rem}\label{inj-mlos-okj}
\begin{itemize}
\item[(i)] If $l=\beta \gamma \in {\mathbb N},$ then the operator $ (-{\mathcal A})_{\gamma \beta}$ in the formulation of problem (FP)$_{\beta}$ can be replaced with the operator $ (-{\mathcal A})^{l}$ therein; cf. (\ref{qwer-zorbinjo}).
\item[(ii)]
Suppose that the operator $C_{1}$ is injective. Then we can simply prove that $ (S_{\gamma,\zeta}(t))_{t\geq 0}$ is a $\zeta$-times integrated $C_{1}$-semigroup
with a subgenerator
$-(-{\mathcal A})_{\gamma},$ which implies \cite{catania} that the integral generator of $ (S_{\gamma,\zeta}(t))_{t\geq 0}$ is
$-C_{1}^{-1}(-{\mathcal A})_{\gamma}C_{1}=-(-{\mathcal A})_{\gamma}.$ A similar statement holds in the case that
$\gamma=1/2,$ which is further discussed in the following theorem.
\end{itemize}
\end{rem}

\begin{thm}\label{5511alm-MLO}
The limit contained in the
expression
\begin{equation}\label{1/2ce-MLO}
S_{1/2}(t)x:=\frac{1}{\pi}\lim \limits_{N\rightarrow \infty}\int
\limits^{N}_{0}\sin \bigl(t\sqrt{\lambda}\bigr)\bigl(\lambda -{\mathcal A}
\bigr)^{-1}C_{1}x\, d\lambda,\quad t>0,
\end{equation}
exists in $L(X)$ for every $x\in X.$ Put $S_{1/2}(0):=C_{1}.$
Then the family $\{S_{1/2}(t) : t>0\}$ is equicontinuous, there exists a strongly analytic operator family $({\bf S}_{1/2}(t))_{t\in \Sigma_{\varphi_{1/2}}}$ such that ${\bf S}_{1/2}(t)=S_{1/2}(t),$ $ t>0$
and the following holds:
\begin{itemize}
\item[(i)] ${\bf S}_{1/2}(t){\bf S}_{1/2}(s)={\bf S}_{1/2}(t+s)C_{1}$ for all $t,\ s\in \Sigma_{\varphi_{1/2}}.$
\item[(ii)]
$\lim_{t\rightarrow 0,t\in \Sigma_{\varphi_{1/2}-\epsilon}}{\bf S}_{1/2}(t)x=C_{1}x,$ $x\in \overline{D({\mathcal A})},$ $\epsilon \in
(0,\varphi_{1/2}).$
\item[(iii)] ${\bf S}_{1/2}(t)(-{\mathcal A})_{\nu}\subseteq
(-{\mathcal A})_{\nu}{\bf S}_{1/2}(t),$ $t\in \Sigma_{\varphi_{1/2}},$ $\nu \in {{\mathbb C}_{+}}.$
\item[(iv)] If $D({\mathcal A})$ is dense in $X,$ then $(S_{1/2}(t))_{t\geq 0}$ is an equicontinuous analytic
$C_{1}$-regularized semigroup of angle $\varphi_{\gamma}.$ Furthermore, $(S_{1/2}(t))_{t\geq 0}$ is a
$C_{1}$-regularized existence family
with a subgenerator
$-(-{\mathcal A})_{1/2}$
and the supposition $(x,y) \in -(-{\mathcal A})_{1/2}$ implies $(C_{1}x,C_{1}y) \in \hat{{\mathcal A}}_{1/2},$ where
$
\hat{{\mathcal A}}_{1/2}$ is the integral generator of $(S_{1/2}(t))_{t\geq 0};$
otherwise, for every $\zeta>0,$ $(S_{1/2,\zeta}(t))_{t\geq 0}$ is an exponentially equicontinuous, analytic $\zeta$-times integrated
$C_{1}$-regularized semigroup, $(S_{1/2,\zeta}(t))_{t\geq 0}$ is a $\zeta$-times integrated $C_{1}$-existence family
with a subgenerator
$-(-{\mathcal A})_{1/2}$ and the supposition $(x,y) \in -(-{\mathcal A})_{1/2}$ implies $(C_{1}x,C_{1}y) \in \hat{{\mathcal A}}_{1/2}.$
\item[(v)]
Then $R(S_{1/2}(t))\subseteq D_{\infty}({\mathcal A}),$ $t>0$ and, for every
$x\in \overline{D({\mathcal A})},$ the incomplete abstract Cauchy problem\index{incomplete abstract Cauchy problems}
\[ (P_{2}) : \left\{
\begin{array}{l}
u\in C^{\infty}\bigl((0,\infty ): X\bigr),\\
u^{\prime \prime }(t)\in -{\mathcal A} u(t),\ t>0, \\
\lim_{t\rightarrow 0+}u(t)=C_{1}x,\\
\mbox{the set } \{ u(t) : t>0\} \mbox{ is bounded in }X,
\end{array}%
\right.
\]
has a solution $u(t)=S_{1/2}(t)x,$ $t>0.$ Moreover, the
mapping $t\mapsto u(t),$ $t>0$ can be analytically extended to the
sector $\Sigma_{\varphi_{1/2}}$ and, for every $\delta \in
(0,\varphi_{1/2})$ and $j\in {\mathbb{N}}_{0},$ we have that
the set $\{z^{j}u^{(j)}(z) : z\in \Sigma_{\delta}\} $ is bounded in
$X.$
\end{itemize}
\end{thm}

\begin{proof}
First of all, observe that $\varphi_{1/2}=\vartheta/2.$
Applying the
partial integration, (\ref{creso}) and the
equicontinuity of family $\{\lambda^{2}(\lambda -{\mathcal A})^{-2}C_{1} : \lambda>0
\},$ we obtain that the limit contained in (\ref{1/2ce-MLO}) exists and
equals
\begin{align*}
S_{1/2}(t)x=\int
\limits_{0}^{\infty}f(\lambda,t)\bigl(\lambda-{\mathcal A}\bigr)^{-2}C_{1}x\,
d\lambda,\quad t>0,\ x\in X,
\end{align*}
where
$f(\lambda,t)=2\pi^{-1}t^{-2}[\sin(t\sqrt{\lambda})-t\sqrt{\lambda}\cos(t\sqrt{\lambda})]$
for $\lambda>0$ and $t>0.$ As in single-valued case, the change of variables
$x=t\sqrt{\lambda}$ shows that the operator family
$\{S_{1/2}(t) : t> 0\}$ is both equicontinuous and strongly continuous. Let $(x,y)\in {\mathcal A}.$ Then Theorem \ref{C-favini}(i) and an elementary argumentation show that
\begin{align*}
S_{1/2}(t)x-C_{1}x&=\frac{1}{\pi}\lim \limits_{N\rightarrow \infty}\int
\limits^{N}_{0}\sin \bigl(t\sqrt{\lambda}\bigr)\Bigl( \bigl(
\lambda+A\bigr)^{-1}C_{1}x- \lambda^{-1}C_{1}x\Bigr)\, d\lambda\\ &=
\frac{1}{\pi}\lim \limits_{N\rightarrow \infty}\int
\limits^{N}_{0}\frac{\sin
\bigl(t\sqrt{\lambda}\bigr)}{\lambda}\bigl(
\lambda -{\mathcal A}\bigr)^{-1}C_{1}y\, d\lambda .
\end{align*}
Keeping in mind the last equality and the equicontinuity of family $\{S_{1/2}(t) :
t\geq 0\}$, we get that $\lim_{t\rightarrow 0}S_{1/2}(t)x=C_{1}x$ for all $x\in \overline{D({\mathcal A})}.$

Now we proceed as in the proof of \cite[Theorem 2.6(i)]{inc}.
Let $0<\delta'<\delta<\vartheta/2,$ $1/2>\gamma_{0}>\delta/\vartheta$ and
$\theta \in (-\vartheta,(-\delta)/\gamma_{0}).$ Then, for every $\gamma \in (\gamma_{0},1/2),$ we have $\theta \in (-\vartheta,(-\delta)/\gamma)$ and $\gamma>\delta/\vartheta.$
Let $\epsilon \in (0,(\pi -\vartheta)/2)$ be sufficiently small. Define, for every $\gamma \in (\gamma_{0},1/2)$ and $x\in X,$
\begin{align*}
F_{\gamma}(\lambda)x:= \left\{
\begin{array}{l}
\frac{e^{i\theta \gamma}\sin \gamma \pi}{\pi}\int^{\infty}_{0}\frac{v^{\gamma}
( v-e^{i\theta}{\mathcal A} )^{-1}C_{1}x\, dv}{( \lambda e^{i\theta \gamma}+v^{\gamma}\cos \pi \gamma)^{2}+v^{2\gamma}
\sin^{2}\gamma \pi}, \mbox{ if }\arg(\lambda) \in (-\epsilon,(\pi/2)+\delta),\\
\frac{e^{-i\theta \gamma}\sin \gamma \pi}{\pi}\int^{\infty}_{0}\frac{v^{\gamma}
( v-e^{-i\theta}{\mathcal A})^{-1}C_{1}x\, dv}{( \lambda e^{-i\theta \gamma}+v^{\gamma}\cos \pi \gamma)^{2}+v^{2\gamma}\sin^{2}\gamma \pi},\mbox{ if }\arg(\lambda) \in (-(\pi/2)-\delta,\epsilon).
\end{array}
\right.
\end{align*}
If $x\in X$ and $\arg(\lambda) \in (-\epsilon,(\pi/2)+\delta),$ resp., $\arg(\lambda) \in (-(\pi/2)-\delta,\epsilon),$ then
\begin{equation}\label{blues}
\int^{\infty}_{0}e^{-\lambda e^{i\theta \gamma}t}S_{\theta,\gamma}(t)x\, dt=\frac{\sin \gamma \pi}{\pi}\int^{\infty}_{0}\frac{v^{\gamma}\bigl( v-e^{i\theta}{\mathcal A} \bigr)^{-1}C_{1}x}{\bigl( \lambda e^{i\theta \gamma}+v^{\gamma}\cos \pi \gamma\bigr)^{2}+v^{2\gamma}\sin^{2}\gamma \pi}\, dv,
\end{equation}
resp.,
\begin{equation}\label{blues-1}
\int^{\infty}_{0}e^{-\lambda e^{-i\theta \gamma}t}S_{-\theta,\gamma}(t)x\, dt=\frac{\sin \gamma \pi}{\pi}\int^{\infty}_{0}\frac{v^{\gamma}\bigl( v+e^{-i\theta}{\mathcal A} \bigr)^{-1}C_{1}x}{\bigl( \lambda e^{-i\theta \gamma }+v^{\gamma}\cos \pi \gamma\bigr)^{2}+v^{2\gamma}\sin^{2}\gamma \pi}\, dv.
\end{equation}
Furthermore,
\begin{equation}\label{blues-2}
e^{i\theta \gamma}\int^{\infty}_{0}e^{-\lambda e^{i\theta \gamma}t}S_{\theta,\gamma}(t)x\, dt=
e^{-i\theta \gamma}\int^{\infty}_{0}e^{-\lambda e^{-i\theta \gamma}t}S_{-\theta,\gamma}(t)x\, dt,\quad \lambda \in \Sigma_{\epsilon}.
\end{equation}
By (\ref{blues})-(\ref{blues-2}), we deduce that the function $\lambda \mapsto F_{\gamma}(\lambda)x,$ $\lambda \in \Sigma_{(\pi/2)+\delta}$
is well defined, analytic and bounded by $\text{Const}_{\delta'}|\lambda|^{-1}$ on sector $\Sigma_{(\pi/2)+\delta'}$ ($x\in X$), as well as
\begin{align}\label{jednakost-MLOs-sad}
S_{\gamma}(z)x=\frac{1}{2\pi i}\int_{\Gamma_{\delta',z}}e^{\lambda z}F_{\gamma}(\lambda)x\, d\lambda,\quad x\in X,\ z\in \Sigma_{\delta'},\ \gamma \in (\gamma_{0},1/2),
\end{align}
where $\Gamma_{\delta',z}:= \Gamma_{\delta',z,1} \cup \Gamma_{\delta',z,2} ,$ $\Gamma_{\delta',z,1}:=\{re^{i((\pi/2)+\delta')} : r\geq
|z|^{-1}\} \cup \{|z|^{-1}e^{i\vartheta} : \vartheta \in
[0,(\pi/2)+\delta']\}$ and $\Gamma_{\delta',z,2}:=\{
re^{-i((\pi/2)+\delta')} : r\geq |z|^{-1}\} \cup
\{|z|^{-1}e^{i\vartheta} : \vartheta \in [-(\pi/2)-\delta',0]\}$ are oriented counterclockwise.
The dominated convergence theorem shows that, for every $x\in X$ and $ z\in \Sigma_{\delta'},$
\begin{align*}
\lim \limits_{\gamma \rightarrow \frac{1}{2}-}S_{\gamma}(z)x&=\frac{e^{i\theta/2}}{2\pi^{2}i}\int_{\Gamma_{\delta',z,1}}e^{\lambda z}
\int^{\infty}_{0}\frac{v^{1/2}\bigl( v-e^{i\theta}{\mathcal A} \bigr)^{-1}C_{1}x}{\lambda^{2}e^{i\theta}+v}\, dv\, d\lambda
\\ & +\frac{e^{-i\theta/2}}{2\pi^{2}i}\int_{\Gamma_{\delta',z,2}}e^{\lambda z}
\int^{\infty}_{0}\frac{v^{1/2}\bigl( v-e^{-i\theta}{\mathcal A} \bigr)^{-1}C_{1}x}{\lambda^{2}e^{-i\theta}+v}\, dv\, d\lambda
\\ & :={\bf S}_{1/2}(z)x.
\end{align*}
Define $F_{1/2}(\lambda)$ by replacing the number $\gamma$ with the number $1/2$ in definition of $F_{\gamma}(\lambda).$ Then, for every $x\in X,$ the function $\lambda \mapsto F_{1/2}(\lambda)x,$ $\lambda \in \Sigma_{(\pi/2)+\delta}$
is well defined and analytic on $ \Sigma_{(\pi/2)+\delta};$  furthermore,
for each $q\in \circledast$ there exists $r_{q}\in \circledast$ such that $q(F_{1/2}(\lambda)x)\leq r_{q}(x)
\text{Const}_{\delta'}|\lambda|^{-1},$ $\lambda \in \Sigma_{(\pi/2)+\delta'},$ $x\in X$ (\cite{inc}). Define $({\bf S}_{1/2}(z))_{z\in \Sigma_{\vartheta/2}}\subseteq L(X)$ by ${\bf S}_{1/2}(z)x:=\lim_{\gamma \rightarrow \frac{1}{2}-}S_{\gamma}(z)x,$
$z\in \Sigma_{\vartheta/2},$ $x\in X$; this operator family is equicontinuous on any proper
subsector of $\Sigma_{\vartheta/2}$ and
satisfies additionally that
the mapping $z\mapsto {\bf S}_{1/2}(z)x,$ $z\in \Sigma_{\vartheta/2}$ is analytic for all $x\in X.$ Letting $ \gamma \rightarrow \frac{1}{2}-$ in (\ref{jednakost-MLOs-sad}), it is not difficult to prove that
\begin{align*}
{\bf S}_{\frac{1}{2}}(z)x=\frac{1}{2\pi i}\int_{\Gamma_{\delta',z}}e^{\lambda z}F_{\frac{1}{2}}(\lambda)x\, d\lambda,\quad x\in X,\ z\in \Sigma_{\delta'},
\end{align*}
so that the proof of \cite[Theorem 2.6.1]{a43} implies
\begin{align}\label{jednakost-MLOs-sad-sa}
\int \limits^{\infty}_{0}e^{-\lambda t}{\bf S}_{\frac{1}{2}}(t)x\, dt=F_{\frac{1}{2}}(\lambda)x,\quad x\in X,\ \lambda>0.
\end{align}
On the other hand, the arguments used in the proof of \cite[Theorem 5.5.2, p. 133]{martinez} show that
\begin{align}\label{jednakost-MLOs-sad-sade}
\int \limits^{\infty}_{0}e^{-\lambda t}S_{\frac{1}{2}}(t)x\, dt=\frac{1}{\pi}\int\limits^{\infty}_{0}\frac{\sqrt{\nu}}{\lambda^{2}+\nu}
\bigl(\nu-{\mathcal A}\bigr)^{-1}C_{1}x\, d\nu
=F_{\frac{1}{2}}(\lambda)x,\quad x\in X,\ \lambda>0.
\end{align}
Using the uniqueness theorem for the Laplace transform, we obtain from (\ref{jednakost-MLOs-sad-sa})-(\ref{jednakost-MLOs-sad-sade})
that ${\bf S}_{1/2}(t)=S_{1/2}(t),$ $ t>0.$ Now the proofs of (i)-(iii) become standard and therefore omitted.

For simplicity, we assume that ${\mathcal A}$ is densely defined in (iv). Then the only non-trivial thing that should be proved is that the supposition $(x,y) \in -(-{\mathcal A})_{1/2}$ implies $(C_{1}x,C_{1}y) \in \hat{{\mathcal A}}_{1/2}.$ So, let $(x,y) \in -(-{\mathcal A})_{1/2},$ i.e., $C_{1}x=-(-{\mathcal A})_{C_{1}}^{-1/2}y.$ 
A similar line of reasoning as in the proof of identity \cite[(51), p. 489]{TAJW} shows that
\begin{align*}
C_{1}\int^{\infty}_{0}e^{-\lambda t}S_{\gamma}(t)y \, dt=C_{1}\bigl( - {\mathcal A}\bigr)_{C_{1}}^{-\gamma}y-\lambda \int^{\infty}_{0}e^{-\lambda t}S_{\gamma}(t)y\, dt,\quad \lambda>0,\ \gamma \in (0,1/2).
\end{align*}
Taking the limits of both sides of previous equality when $\gamma \rightarrow 1/2-,$ we get that
\begin{align*}
C_{1}\int^{\infty}_{0}e^{-\lambda t}S_{1/2}(t)y \, dt=C_{1}\bigl( - {\mathcal A}\bigr)_{C_{1}}^{-1/2}y-\lambda \int^{\infty}_{0}e^{-\lambda t}S_{1/2}(t)y\, dt,\quad \lambda>0.
\end{align*}
Then the uniqueness theorem for the Laplace transform simply implies that
$$
S_{1/2}(t)C_{1}x-C_{1}^{2}x=\int^{t}_{0}S_{1/2}(s)C_{1}y\, ds,\quad t\geq 0,
$$
as claimed.

Now we will prove (v) by slightly modifying the arguments used in the corresponding part of proof of \cite[Theorem 2.6(i)]{inc}.
In order to do that, we will first show that
for each $x\in X$ we have $S_{1/2}^{\prime \prime}(t)x\in -{\mathcal A}S_{1/2}(t)x,$ $t>0.$ Fix temporarily an element $x\in X.$ Owing to Theorem \ref{thm-mlo-fgrp}(v)
and (\ref{metallica-MLOS}), cf. also Remark \ref{inj-mlos-okj}(i), we have that
\begin{align*}
D_{-}^{\frac{1}{\gamma}}S_{\gamma}^{\prime}(t)x\in -{\mathcal A}S_{\gamma}(t)x,\quad t>0,
\end{align*}
i.e.,
\begin{align*}
\frac{d^{2}}{dt^{2}}\int_{0}^{\infty}g_{3-\frac{1}{\gamma}}(s)S_{\gamma}^{\prime}(t+s)x\, ds\in {\mathcal A}S_{\gamma}(t)x,\quad t>0,\ \gamma \in (\gamma_{0},1/2).
\end{align*}
Therefore,
$$
\int_{0}^{\infty}g_{3-\frac{1}{\gamma}}(s)S_{\gamma}^{\prime \prime \prime}(t+s)x\, ds\in {\mathcal A}S_{\gamma}(t)x,\quad t>0,\ \gamma \in (\gamma_{0},1/2).
$$
Applying the partial integration, we get
$$
\int_{0}^{\infty}g_{4-\frac{1}{\gamma}}(s)S_{\gamma}^{(iv)}(t+s)x\, ds \in -{\mathcal A}S_{\gamma}(t)x,\quad t>0,\ \gamma \in (\gamma_{0},1/2).
$$
The dominated convergence theorem yields by letting $\gamma \rightarrow 1/2-$ that
$$
\int^{\infty}_{0}sS_{1/2}^{(iv)}(t+s)x\, ds \in -{\mathcal A}S_{1/2}(t)x,\quad t>0,
$$
which clearly implies after an application of integration by parts that $S_{1/2}^{\prime \prime}(t)x\in -{\mathcal A}S_{1/2}(t)x,$ $t>0,$
as claimed. By (ii), the function $u(t)= S_{1/2}(t)x,$ $t>0$ is a solution of problem $(P_{2})$ for $x\in \overline{D({\mathcal A})}.$
Furthermore, we obtain by induction that $S_{1/2}^{(2n)}(t)x\in (-1)^{n}{\mathcal A}^{n}S_{1/2}(t)x,$ $t>0,$ $n\in {\mathbb N},$ $x\in X,$
so that $R(S_{1/2}(t))\subseteq D_{\infty}({\mathcal A}),$ $t>0.$ The proof of the theorem is thereby complete.
\end{proof}

\begin{example}\label{miao-milica-integrisane}
Examples of exponentially bounded integrated semigroups generated by multivalued linear operators can be found in \cite[Section 5.3, Section 5.8]{faviniyagi} (cf. also \cite{arendt-favini}); multivalued matricial operators on product spaces can also generate exponentially bounded degenerate integrated semigroups (see e.g. \cite[Example 3.2.24]{knjigaho} for single-valued case).
It is also worth noting that multivalued linear operators whose resolvent sets contain certain exponential regions \cite{a22} have been considered in \cite[Example 3.2.11(i)]{FKP} as generators of degenerate local once integrated
semigroups. All these examples can serve one to provide possible applications of Theorem \ref{thm-mlo-fgrp} and Theorem \ref{5511alm-MLO}.
\end{example}

\begin{example}\label{miao-milica-konvolucione}
Assume that $(M_p)$ is a sequence of positive real numbers
such that $M_0=1$ and that the following conditions are fulfilled:
\[
\tag{M.1}
M_p^2\leq M_{p+1} M_{p-1},\;\;p\in\mathbb{N},
\]
\begin{gather*}\tag{M.2}
M_p\leq AH^p\sup_{0\leq i\leq p}M_iM_{p-i},\;\;p\in\mathbb{N},\mbox{ for some }A,\ H>1,
\end{gather*}
and
\[\tag{M.3}
\sup_{p\in\mathbb{N}}\sum_{q=p+1}^{\infty}\frac{M_{q-1}M_{p+1}}{pM_pM_q}<\infty .
\]
For example, we can take the Gevrey sequence $(M_{p}\equiv p!^s)$ with $s>1.$
Set
$m_p:=\frac{M_p}{M_{p-1}}$, $p\in\mathbb{N}$ and
$\omega(z):=\prod_{i=1}^{\infty}\bigl(1+\frac{iz}{m_p}\bigr)$, $z\in\mathbb{C}$. Suppose that
there exist constants $l>0$ and $\omega>0$ satisfying that
$
RHP_{\omega}\equiv \{ \lambda \in {\mathbb C} : \Re \lambda >\omega\} \subseteq\rho({\mathcal A})$ and the operator family $\{e^{-M(l|\lambda|)}R(\lambda : {\mathcal A}) \, | \, \lambda\in RHP_{\omega}\}\subseteq L(X)$ is equicontinuous (cf. \cite{knjigah}-\cite{knjigaho} for a great number of such examples with ${\mathcal A}$ being single-valued, and \cite[Example 3.25]{catania} for purely multivalued linear case, with $X$ being a Fr\' echet space). Let $\bar{\omega}>\omega.$ Then there exists a sufficiently large number $n \in {\mathbb N}$ such that the expression
\begin{equation}\label{spiritual}
S(t):=\frac{1}{2\pi i}\int\limits_{\bar{\omega}-i\infty}^{\bar{\omega}+i\infty}e^{\lambda t}\frac{R(\lambda\!:\!{\mathcal A})}{\omega^{n}(i\lambda)}\,d\lambda,\quad t\geq 0
\end{equation}
defines an exponentially equicontinuous $C\equiv S(0)$-regularized semigroup $(S(t))_{t\geq 0}$ with a subgenerator ${\mathcal A}$ (cf. Lemma \ref{integracija-tricky}, Theorem \ref{C-favini}(i) and the proof of \cite[Theorem 3.6.4]{knjigah}). It is not difficult to prove that
$(\lambda -{\mathcal A})^{-1}Cf=\int^{\infty}_{0}e^{-\lambda t}S(t)f\, dt,$ $\Re \lambda >\bar{\omega},$ $f\in X,$ so that
 Theorem \ref{thm-mlo-fgrp} and Theorem \ref{5511alm-MLO} can be applied with the operator ${\mathcal A}$ replaced with the operator ${\mathcal A}-\bar{\omega}$ therein. Observe that,  even in single-valued linear case, the operator $C$ need not be injective because our
assumptions do not imply that ${\mathcal A}=A$ generates an ultradistribution semigroup of Beurling class (cf. \cite{FKP} for the notion).
\end{example}

In this paper, we will not discuss the generation of degenerate fractional regularized resolvent families by the negatives of constructed fractional powers. For more details, cf. \cite[Section 3]{TAJW} and \cite[Remark 2.9.49]{knjigah}.

At the end of paper, we would like to observe that the assertions of \cite[Theorem 2.4-Theorem 2.6]{inc} can be formulated in the multivalued linear operators setting. For applications, the most important is the following case:
$X$ is a Banach space, $\Sigma_{\vartheta} \cup B_{d} \subseteq \rho({\mathcal A}),$ there exist finite numbers $M_{1}\geq 1$ and $\nu \in (0,1]$ such that (\ref{fracvt-power-t-power}) holds with the operator ${\mathcal A}$ and number $\beta$ replaced with $-{\mathcal A}$ and $\nu$ therein; see \cite[Chapter III, Chapter VI]{faviniyagi} for a great number of concrete examples. Define the operators $S_{\gamma}(\cdot)$ as before.
Then we may conclude the following:
\begin{itemize}
\item[(i)] Suppose that $\beta \gamma >1-\nu.$ Then $(S_{\gamma}(t))_{t\in \Sigma_{\varphi_{\gamma}}}$ is an analytic semigroup of growth order $\frac{1-\nu}{\gamma}$ (cf. \cite{tanakaa} or \cite[Definition 2.1]{inc} for the notion). Denote by $\Omega_{\theta,\gamma},$ resp. $\Psi_{\gamma},$
% u slucaju $m=-1,$ angularni set sadrzi zatvorenje od D(A)!!!
 the continuity set of $(S_{\gamma}(te^{i\theta}))_{t>0},$ resp.
$(S_{\gamma}(t))_{t\in \Sigma_{\varphi_{\gamma}}}.$ Then $\overline{D({\mathcal A})}\subseteq \Psi_{\gamma}$ and, for
every $x\in \Omega_{\theta,\gamma},$ the incomplete abstract Cauchy
inclusion $(FP_{\beta}),$ with $C_{1}=I,$ has a solution $u(t)=S_{\gamma}(te^{i\theta})x,$ $t>0,$ which can be
analytically extended to the sector $\Sigma_{\varphi_{\gamma}-|\theta|}.$ If, additionally, $x\in \Psi_{\gamma},$ then for
every $\delta \in (0,\varphi_{\gamma})$ and $j\in
{\mathbb{N}}_{0},$ we have that the set $\{z^{j}u^{(j)}(z) : z\in
\Sigma_{\delta}\} $ is bounded in $X.$
\item[(ii)] Suppose that $1/2<
\nu<1.$ Then the incomplete abstract Cauchy problem\index{incomplete abstract Cauchy problems}
$(P_{2}) ,$ with $C_{1}=I,$
has a solution $u(t),$ $t>0$ for all $x\in D({\mathcal A}).$ Moreover, the
mapping $t\mapsto u(t),$ $t>0$ can be analytically extended to the
sector $\Sigma_{\varphi_{1/2}}$ and, for every $\delta \in
(0,\varphi_{1/2})$ and $j\in {\mathbb{N}}_{0},$ we have that
the set $\{z^{j}(1+|z|^{2\nu-2})^{-1}u^{(j)}(z) : z\in \Sigma_{\delta}\} $ is bounded in
$X.$
\end{itemize}

It is very non-trivial to find some necessary and sufficient conditions ensuring the uniqueness of solutions of problems $(FP_{\beta})$
and $(P_{2}) ;$ cf. also \cite{inc}. This is an open problem we would like to address to our researchers.

\end{document}